\documentclass[11pt, a4paper]{amsart}
\usepackage[utf8]{inputenc}
\usepackage[english]{babel}
\usepackage[margin=2.5cm]{geometry}
\usepackage{float}
\usepackage{tikz,tikz-cd}
\usepackage{hyperref}
\usepackage{mathrsfs}
\usepackage{enumitem}
\usepackage{graphicx}
\usepackage{amsmath,amsfonts,amssymb,mathtools}
\usepackage{subcaption}
\usepackage{hyperref}
\usepackage{soul}
\usetikzlibrary{calc,patterns,arrows,shapes.arrows,intersections}
\usetikzlibrary{decorations}

\newtheorem{thm}{Theorem}[section]
\newtheorem*{thm*}{Theorem}
\newtheorem{prop}[thm]{Proposition}
\newtheorem*{prop*}{Proposition}
\newtheorem{lem}[thm]{Lemma}
\newtheorem{cor}[thm]{Corollary}

\theoremstyle{remark}
\newtheorem{rem}[thm]{Remark}
\newtheorem{example}[thm]{Example}

\newtheorem{defn}[thm]{Definition}

\newcommand{\origraph}{\textbf{OGraph}}
\newcommand{\rmod}{R\text{-}\mathbf{Mod}}

\newcommand{\MM}{\mathrm{M}}
\newcommand{\mat}{\mathrm{M}^o}
\newcommand{\mor}{\mathcal{M}}

\newcommand{\hasse}{\mathtt{F}}
\newcommand{\orient}{\tG_\mathfrak{o}}

\newcommand{\oo}{\mathfrak{o}}

\newcommand{\bN}{\mathbb{N}}

\newcommand{\bS}{\mathbb{S}}
\newcommand{\bZ}{\mathbb{Z}}

\newcommand{\bF}{\mathbb{F}}
\newcommand{\cF}{\mathcal{F}}

\newcommand{\tG}{{\tt G}}
\newcommand{\tH}{{\tt H}}
\newcommand{\tA}{{\tt A}}
\newcommand{\tC}{{\tt C}}
\newcommand{\tP}{{\tt P}}
\newcommand{\tK}{{\tt K}}
\newcommand{\tL}{{\tt L}}

\newcommand{\e}{\varepsilon}
\definecolor{bunired}{rgb}{0.8, 0.0, 0.0}

\title{Monotone cohomologies and oriented matchings}
\author{Luigi Caputi}
\author{Daniele Celoria}
\author{Carlo Collari}
\date{}

\begin{document}

\begin{abstract}
In this paper, we extend the definition of cohomology associated to monotone graph properties, to encompass twisted functor coefficients.
We introduce oriented matchings on graphs, and focus on their (twisted) cohomology groups.
We characterise oriented matchings in terms of induced free-flow pseudoforests, and explicitly determine the homotopy type of the associated simplicial complexes. Furthermore, we provide a connection between the cohomology of oriented matchings with certain functor coefficients, and the recently defined multipath cohomology. Finally,  we define a further oriented homology for graphs and interpret it as a count of free-flow orientations.
\end{abstract}
\maketitle

\section{Introduction}

In recent years, the study of graphs by means of techniques rooted in combinatorial algebraic topology has developed at a very rapid pace. One of the main avenues of research focuses on the interplay between graphs and simplicial complexes; indeed, there are many distinct natural ways to associate a simplicial complex to a graph. Combinatorial, homotopical, and homological properties of these complexes encode interesting information about the graph. 
The study of these simplicial complexes is now an established and active area of research~\cite{MR2022345, Kozlov, Jonsson}, with deep connections with other areas of mathematics -- see \cite[Chapter 1]{Jonsson} for an overview.

In this paper we focus on a special class of simplicial complexes associated to certain graph matchings. As most complexes related to matchings on graphs, they arise as a special case of \emph{monotone complexes}; that is simplicial complexes of subgraphs defined by a property which is closed under edge removal. Such properties are collectively called \emph{monotone properties}.
Monotone complexes have been extensively studied, both from a combinatorial and topological perspective~\cite{Jonsson}. 

In most cases, investigations of monotone complexes have been carried out by applying standard homological methods.
We start our analysis with the observation that, quite often, more refined invariants arise when considering twisted coefficients. Paralleling well-known constructions in topology (\emph{cf}.~\cite{quillenI} and \cite[Theorem~16.2.3]{richter2020categories}), our interest lies in extending these  constructions to 
functors coefficient.
To this end, using techniques borrowed from the recently defined poset homology~\cite{chandler2019posets, primo} -- for a given monotone property $\mathscr{M}$ on a graph~$\tG$ and a certain functor $\cF_A$ (depending on an algebra $A$) -- we construct a cohomology theory $\mathrm{H}_{\mathscr{M}}(\tG;A)$ which we call \emph{monotone cohomology}.
This approach provides a novel and comprehensive tool to study monotone complexes.
By varying the choice of monotone property, one can realise, for example, the chromatic homology  of graphs~\cite{HGRong} and the multipath cohomology of directed graphs \cite{primo} as monotone cohomologies.

We focus our approach on the special case of \emph{oriented matchings}, and show that their associated monotone cohomology can be effectively computed. 
More precisely, for any oriented graph~$\tG$, we define the oriented matching complex $\mat(\tG)$; this is obtained by considering a subset of matchings on the face poset of $\tG$ that are ``compatible'' with the given orientation. 
Oriented matchings are naturally related to pseudoforests endowed with a specific type of orientations, which we call \emph{free-flow} (\emph{cf.}~Definition~\ref{def:freeflow}).
As a first result, we provide a complete characterisation of these matchings in terms of the subgraphs they induce, in analogy with \cite{KozlovTrees,chari2000discrete}.  
\begin{prop}
Let $\tG$ be an oriented graph.
Simplices in $\mat(\tG)$ are in bijection with rooted spanning pseudoforests in $\tG$ whose induced orientation is free-flow.
\end{prop}

The homotopical and combinatorial properties of the oriented matching complexes, turn out to be tightly intertwined. Indeed, the homotopy type of $\mat(\tG)$ is completely determined by simple quantities associated to the orientation of $\tG$. The explicit description that we obtain exhibits these complexes as iterated suspensions, whose parameters are related to the graph's indegrees (\emph{cf.}~Proposition~\ref{prop:homotopy of or mat}).
\begin{prop}
For a connected  oriented graph $\tG$, the complex $\mat(\tG)$ is either contractible or homotopic to a wedge of spheres. More precisely,  if a vertex in $\tG$ has indegree $1$, then $\mat(\tG)$ is contractible. Otherwise, we have the homotopy equivalence
\[\mat(\tG) \simeq \bigvee^{q} \bS^{N -1} \ ,\]
where $q$ and $N$ are natural numbers depending on the indegrees of vertices in $\tG$.
\end{prop}

Our proof of the above result makes use of a correspondence between matchings on graphs and multipaths -- some special subgraphs satisfying a monotone property \cite{turner}.
A similar correspondence between oriented matchings and multipaths can be used to shed light on the simplicial structure of $\mat(\tG)$, allowing us to obtain the following result:

\begin{cor}
Oriented matching complexes are either strongly shellable or contractible. \\ In particular, all non-contractible matching complexes are Cohen-Macaulay.
\end{cor}
 
For an oriented graph $\tG$ and a commutative algebra $A$, let $C_{o}^*(\tG;A)$ and $C_{\mu}^*(\tG;A)$ be the cochain complexes computing the cohomologies associated to the oriented matching and multipath complexes, respectively. The relation between these two complexes is not direct. 
Rather, to an oriented graph $\tG$ we associate its ``source resolution'' $\tG_{sr}$ (Definition~\ref{def:G_sr}), then the relation is established as follows:

\begin{thm}
The association $\tG \mapsto \tG_{sr}$ yields the following isomorphism of cochain complexes
\[
C_{o}^*(\tG;A)\cong C_{\mu}^*(\tG_{sr};A)\otimes A^{\otimes s}
\]
where $s$ is the number of vertices of indegree $0$ in $\tG$.
\end{thm}

Another interesting aspect of the oriented matching complexes $\mat$, is that they provide a decomposition in subcomplexes of the matching complex of $\tG$.
In the spirit of~\cite{caputi2022categorifying}, we combine them in a single object;
given an oriented graph $\tG_{\oo}$, we consider all possible orientations on the underlying unoriented graph~$\tG$. The choice of $\oo$ provides a natural identification of this set with a Boolean poset. 
We decorate the vertices of this poset with the simplices of the corresponding oriented matching complex.
We can then define $\mathrm{OH}^*(\orient)$, the oriented homology of $\tG$ with respect to the orientation $\oo$.
Surprisingly, it turns out that the structure of this homology is very simple, and tightly related to free-flow orientations:

\begin{thm}
Let $\tG$ be a connected unoriented graph, and let $\oo$ be an orientation on $\tG$. Then,  generators of $\mathrm{OH}^*(\orient)$ are in bijection with free-flow orientations on $\tG$.
\end{thm}

Despite the fact that the dimension of ${\rm OH}(\tG_\oo)$ depends only on the underlying unoriented graph, the grading in ${\rm OH}$ can distinguish between different orientations.

\subsection*{Acknowledgements}
LC~acknowledges support from the \'{E}cole Polytechnique F\'{e}d\'{e}rale de Lausanne via a collaboration agreement with the University of Aberdeen.
DC~was partially supported by the European Research Council (ERC) under the EU Horizon 2020 research and innovation programme (grant agreement No 674978), and by Hodgson-Rubinstein's ARC grant DP190102363 ``Classical And Quantum Invariants Of Low-Dimensional Manifolds''.
CC~is supported by the MIUR-PRIN project 2017JZ2SW5.

\section{Monotone cohomologies of oriented graphs}\label{sec:cohomologies}

Among the most common approaches in defining (co)homology theories of graphs is to first associate to a graph~$\tG$ a simplicial complex~$S(\tG)$, and then to compute the classical (co)homology groups of  $S(\tG)$. Similarly, one can associate to $S(\tG)$ a poset; for instance, one can consider its face poset ${ F}(S(\tG))$. Then, one  can apply a (co)homology theory of posets. The simplest approach, given by computing poset homology groups as in~\cite{Wachs}, yields the simplicial cohomology groups of $S(\tG)$. Since a poset can be seen as a category, we can borrow results and methods from category theory; the following diagram illustrates the idea: 
\[
\begin{tikzcd}
 & &\mathbf{Cat} \arrow[dddrr,bend left,  ""]& &  \\
 & & \mathbf{Poset}\arrow[u, ]\arrow[ddrr,  ""] && \\
 & &\mathbf{Simp}\arrow[u, ]\arrow[drr]&& \\
\origraph \arrow[rrrr, "\text{Homologies of oriented graphs}"]\arrow[urr, ""]\arrow[uurr, ""] \arrow[uuurr, bend left, ""] & & & & \mathbf{Ab} 
\end{tikzcd}
\]
The unmarked arrows pointing towards ${\bf Ab}$ represent homology theories for simplicial complexes, posets, and categories, respectively.
In this section  we introduce the so-called ``poset homology'' as in \cite{chandler2019posets,primo}, and specialise this construction to certain posets associated to oriented graphs.

\subsection{Poset homology}\label{posethomology}
We start by reviewing the definition of poset homology for a finite poset~$P$ with coefficients in a functor~$\mathcal{F}$. We remark here that this construction is related to, but not the same as, the classical poset homology (see, \emph{e.g.}~\cite{Wachs}) which is defined as the homology of the nerve associated to the poset. We refer to \cite{chandler2019posets, primo} for more general expositions on the topic. 

For a poset~$(P,<)$, let $\prec$ denote the \emph{covering relation} associated to $<$, \emph{i.e.}~$x\prec y$ if and only if~$x < y$ and there is no $z$ such that $x< z <y$. 

We say that $P$ is \emph{ranked} if there is a rank function $\ell\colon P\to \bN$ such that $x\prec y$ implies $\ell(y)=\ell(x)+1$. We say that  $P$ is \emph{squared} if, for each triple $x,y,z\in P$ such that $z\prec y \prec x$, there is a unique $y' \neq y$ such that $z\prec y' \prec x$. Such elements $z,y,y',x$, together with their covering relations in $P$, will be called a \emph{square}.
In what follows, we will assume all posets to be ranked and squared.

\begin{example}
Recall that a \emph{regular} CW-complex is a CW-complex for which all attaching maps are homeomorphisms.
A \emph{CW-poset} (\emph{i.e.}~the collection of the cells in a regular CW complex, ordered by containment~\cite[Definition~2.1 \& Proposition~3.1]{BJOR}) is ranked and squared; the rank function is given by the dimension of the cells. In particular, the face poset $F(X)$ of a simplicial complex $X$ is ranked and squared.
\end{example}

A finite poset $(P,<)$ can be  seen as a (small) category $\mathbf{P}$; the set of objects of {\bf P} is the set $P$, and there is a unique morphism $x\to y$ if and only if $x\leq y$. 
Functors on the category associated to the poset~$P$ preserve commutative squares:
For each $x, z\in P$, $x\leq z$, there is a unique mapping $f_{x,z}\colon x\to z$ in the  category {\bf P}.
Assume there is a square between $x$ and~$z$; the existence of such a square implies that $f_{x,z}$ factors as follows
\[ f_{x,z} = f_{y,z} \circ f_{x,y} = f_{y',z} \circ f_{x,y'} \ .\]
Therefore, given a covariant functor $\mathcal{F}: {\bf P} \to {\bf C}$, we must have:
\[\mathcal{F}(f_{y,z}) \circ \mathcal{F}(f_{x,y}) =\mathcal{F}(f_{y,z} \circ f_{x,y}) =\mathcal{F}(f_{x,z})= \mathcal{F}(f_{y',z} \circ  f_{x,y'}) = \mathcal{F}(f_{y',z}) \circ \mathcal{F}(f_{x,y'}) \ . \]
In other words, all functors preserve the commutativity of the squares in $P$.

Let $\bZ_2$ be the cyclic group on two elements.
\begin{defn}\label{def:sign_ass}
	A \emph{sign assignment} on a poset $(P,<)$ is an assignment of elements $\e({x,y})\in \bZ_2$ to each pair of elements $x,y\in P$ with $x \prec y$, such that 
	\begin{equation*}\label{eq:signassign}
		\e({x,y}) + \e({y,z}) \equiv \e({x,y'}) + \e({y',z}) + 1 \mod 2
	\end{equation*}
	holds for each square  $x \prec y,~y' \prec z$. 
\end{defn}

In general, the existence of a sign assignment on a poset~$P$ depends on the topology of a certain CW-complex associated to~$P$ -- see, \emph{e.g.}~\cite[Section~3.2]{primo}, \cite[Section~5]{Putyra}.

\begin{rem}\label{rem:Boolean sings}
Every CW-poset admits a sign assignment, which is unique up to (a suitable notion of) isomorphism -- see, \emph{e.g.}~\cite[Section~4]{chandler2019posets}.
\end{rem}

We can now recall the definition of  poset homology of a poset~$P$ with coefficients in a functor~$\cF$.

Let ${\bf A}$ be an Abelian category -- such as the category of left modules on a commutative ring~$R$ --  $P$ a ranked squared poset with rank function~$\ell$, and $\e$ a sign assignment on $P$.
Given a covariant functor 
$ \mathcal{F}\colon{\bf P} \to {\bf A}$,
we define the cochain groups 
\begin{equation} 
C^n_{\mathcal{F}}(P) \coloneqq \bigoplus_{\substack{x\in P\\ \ell(x) = n}}  \mathcal{F}(x),
\end{equation}
and the differentials 
\begin{equation}\label{eq:diff}
    d^{n}=d^{n}_{\mathcal{F}} \coloneqq \sum_{\substack{x\in P\\  \ell(x) = n}}\ \sum_{\substack{x'\in P\\ x \prec x'}} (-1)^{ \e(x,x')} \mathcal{F}(x \prec  x') \ . 
\end{equation}
% \begin{equation} 
% C^n_{\mathcal{F}}(P) \coloneqq \bigoplus_{\tiny\begin{matrix}
% 		x\in P\\
% 		\ell(x) = n
% \end{matrix}}  \mathcal{F}(x),
% \end{equation}
% and the differentials 
% \begin{equation}\label{eq:diff}
%     d^{n}=d^{n}_{\mathcal{F}} \coloneqq \sum_{\tiny\begin{matrix}
% 		x\in P\\
% 		\ell(x) = n
% \end{matrix}}\ \sum_{\tiny\begin{matrix}
% 		x'\in P\\
% 		x \prec x'
% \end{matrix}} (-1)^{ \e(x,x')} \mathcal{F}(x \prec  x') \ . 
% \end{equation}

With these definitions in place, we can state one of the main results from~\cite{chandler2019posets} and~\cite{primo}. 

\begin{thm}\label{teo: general cohom}
Let ${\bf A}$ be an Abelian category, $P$  be a ranked squared poset, and $\e$ be a sign assignment on $P$. Then, for any $n\in\mathbb{N}$ and any functor $\mathcal{F}\colon{\bf P} \to {\bf A}$ we have $d^{n+1} \circ d^{n} \equiv 0$. In particular, $(C^*_{\mathcal{F}}(P), d^*)$ is a cochain complex.
\end{thm}

The differentials $d^n$, and therefore the cochain complexes,  depend \emph{a priori} upon the choice of the sign assignment~$\e$.  In the cases of interest to us, that is  CW-posets,  the choice of the sign assignment does not affect the isomorphism type of the cochain complexes~$(C^*_{\mathcal{F}}(P), d^*)$ -- see, for instance, \cite[Corollary~3.18]{primo}.

\subsection{Monotone cohomologies of oriented graphs}
Recall that an \emph{unoriented graph} \tG~is a pair of finite sets~$(V,E)$ consisting of: a set of \emph{vertices} $V$, and a set $E$ whose elements, called \emph{edges}, are unordered pairs of {distinct} vertices of $\tG$. All graphs are assumed to be simple.
We will also consider oriented graphs, whose definition we now recall;

\begin{defn}\label{def:origraph}
An \emph{oriented graph} $\tG$ is a pair of finite sets $(V(\tG),E(\tG))$, such that $E(\tG) $ is a subset of~$V(\tG)\times V(\tG) \setminus \{ (v,v)\ |\ v\in V(\tG) \}$, and at most one among $(v,w)$ and $(w,v)$ belongs to~$E(\tG)$.
\end{defn}

By definition, an edge $e$ of an oriented graph $\tG$ is an ordered set of two distinct vertices, say~$e = (v, w)$. The vertex $v$ is the \emph{source} ${\tt s}(e)$ of $e$, while the vertex $w$ is the \emph{target} ${\tt t}(e)$ of $e$. 

A morphism of oriented graphs is a function $\phi\colon V(\tG_1)\to V(\tG_2)$ sending  edges to edges: 
\[
(v,w)\in E(\tG_1)\Rightarrow (\phi(v),\phi(w))\in E(\tG_2) \ ;
\]
observe that a morphism of oriented graphs does not allow collapsing, meaning that $(v,w)\in E(\tG_1)\Rightarrow \phi(v)\neq \phi(w)$. We call \emph{regular} those morphisms of oriented graphs that are also injective as maps of the vertices\footnote{Note that non-collapsing does not imply regular.}. 
Oriented graphs and regular morphisms of oriented graphs form a category  that we denote by $\origraph$. 

\begin{rem}
The results in this section are stated in the category $\origraph$, but everything holds verbatim for other categories of graphs, such as unoriented graphs, directed graphs, quivers, \emph{etc}.
Throughout the rest of the paper, when clear from the context, we will omit the reference to the category of graphs we are using.
\end{rem}

Let $\mathbf{PosFun}$ be the category of tuples $(P,\cF,\mathbf{A},\e)$ consisting of a poset $P$,  a functor $\cF\colon \mathbf{P}\to \mathbf{A}$ with values in an additive category $\mathbf{A}$, and a sign assignment $\e$ on $P$. Then, any functor
\[
\origraph\to \mathbf{PosFun}
\]
associating to each oriented graph~$\tG$ the tuple $(P,\cF,\mathbf{A},\e)$, produces a homology theory of oriented graphs. We can then apply the poset homology construction detailed in the previous section to obtain cohomology groups $\mathrm{H}_*(P;\cF)$. For a general overview on this framework, see~\cite[Section~7]{chandler2019posets}.

\begin{example}
Let $S\colon \origraph\to \mathbf{Simp}$ a functor from oriented graphs to the category of simplicial complexes and simplicial maps.
Let $R$ be a commutative unital ring,  $\rmod$ the category of $R$-modules, and~$\cF$ the constant functor $R$ on the (category associated to the) face poset $\mathbf{F}(S(\tG))$. Then, for each choice of a sign assignment $\e$ on $P$, we have an induced functor~$\tG\mapsto(\mathbf{F}(S(\tG)),R, \mathbf{Ab}, \e) $ whose associated poset homology is the classical simplicial homology of the simplicial complex $S(\tG)$, with coefficient in the ring $R$. 
\end{example}

We want to apply this general machinery  to the case of a poset $P$, whose elements are subgraphs of a given  graph $\tG$. In such case, we can specify a functor $\mathcal{F}_A \colon {\bf P} \to \rmod$ depending on an algebra $A$.

Let $\tG$ be a graph.
Recall that a property $\mathscr{M}$ on the set $SS(\tG)$ of spanning\footnote{That is, a subgraph whose set of vertices is the same as the set of vertices of the whole graph.} subgraphs of $\tG$ is called \emph{monotone} if $$\mathscr{M}(\tH)\Rightarrow \mathscr{M}(\tK)$$  for any subgraph $\tK$ in $SS(\tG)$ obtained from $\tH$ by removing one edge. 
For a given monotone property $\mathscr{M}$ on $\tG$, we define the poset $P_{\mathscr{M}}(\tG)$ whose elements are  the spanning subgraphs of~$\tG$ satisfying  $\mathscr{M}$, endowed with the inclusion relation.

\begin{example}\label{ex:spanning}
For a (non-necessarily oriented) graph $\tG$, the set of spanning graphs $SS(\tG)$, endowed with the inclusion relation, is the poset associated to the monotone property ``being a spanning subgraph''. Note that $SS(\tG)$ is isomorphic (as a poset) to the Boolean poset $(E(\tG),\subset)$.
\end{example}

\begin{example}\label{ex:multipath}
A \emph{multipath} in an oriented graph $\tG$ is a spanning subgraph whose vertices have both indegree and outdegree at most one, and does not contain any cycle -- \emph{cf.}~\cite{turner}.
The \emph{path poset} $P(\tG)$ is the poset given by all multipaths in $\tG$ ordered by inclusion.
\end{example}

\begin{rem}\label{rem:realizability}
For any given monotone property $\mathscr{M}$, the poset $P_{\mathscr{M}}(\tG)$ is squared and downward closed in the Boolean poset  $SS(\tG)$. Moreover, the poset ${P}_{\mathscr{M}}(\tG)$ is a CW-poset, and  in fact simplicial. 
\end{rem}

 From now on, $R$ denotes a commutative ring with identity, and $A$ is a commutative unital $R$-algebra. We now define an explicit functor $\cF_A\colon{\bf P}_{\mathscr{M}}(\tG)\to \rmod$.
 
Given a subgraph $\tH \subseteq\tG$ such that $\mathscr{M}(\tH)$ holds, to each connected component of $\tH$,  we associate a copy of $A$. Then we take their ordered tensor product; to do so, we must fix an order of the connected components of each subgraph in $\mathscr{M}(\tG)$. It can be proved that this choice is immaterial, \emph{cf}.~\cite{HGRong, primo}. For the sake of concreteness, we fix an order on $V(\tG)$, thus inducing an order of the components in each $\tH\in SS(\tG)$, according to the minimal vertex on each component.
If $c_0<\dots<c_k$ is the set of ordered connected components of $\tH$, we define: 
\begin{equation}\label{eq:fun_obj}
	\mathcal{F}_{A}(\tH)\coloneqq  A_{c_1} \otimes_R \cdots 
	\otimes_R  A_{c_k}\ ,\end{equation}
where  all the modules are labelled  by the respective connected component. 

Assume now that $\tH' = \tH\cup e$ and that $\mathscr{M}(\tH')$ holds. Denote by $c_0$,...,$c_{k}$ the ordered components of $\tH$, and by  $c'_0$,...,$c'_{k-1}$ the ordered components of $\tH'$; further assume that the addition of $e$ merges the components $c_i$ and $c_j$. Then, for each $h=0,...,k-1$, there is a natural identification
\begin{equation}\label{eq:identification_components}
	c'_h = \begin{cases} c_h & \text{if } 0\leq h<i \text{ or } i< h < j;\\ c_i \cup e \cup c_j &\text{if }  h =i; \\ c_{h+1} &\text{if }  j\leq h<k.\end{cases}
\end{equation}
for some $0\leq i < j \leq k$. Using this  identification,
 we  define
$ \mu_{\tH\prec \tH'}\colon \mathcal{F}_{A}(\tH) \longrightarrow \mathcal{F}_{A}(\tH')$
 as 
\[ \mu_{\tH\prec \tH'}(a_0 \otimes \cdots \otimes a_k) =  
a_{0} \otimes \cdots \otimes a_{i-1} \otimes a_{i}\cdot a_{j} \otimes a_{i+1} \otimes \cdots \otimes \widehat{a_{j}} \otimes \cdots \otimes a_{k-1} \otimes a_{k} 
	\]
where $ \widehat{a_{j}} $ indicates the removal of $a_{j}$.
Now, assume that  $\tH' = \tH\cup e$,  that $\mathscr{M}(\tH')$ holds, and that, unlike the above case, the addition of $e$ does not affect the number of connected components. In this case, we have a natural identification between connected components in $\tH$ and $\tH'$, which induces an identification ${\rm I}_{\tH\prec \tH'}: \mathcal{F}_{A}(\tH) \to \mathcal{F}(\tH')$.
We can therefore define
\begin{equation}\label{eq:fun_mor}
	\mathcal{F}_{A}(\tH\preceq \tH')\coloneqq
	\begin{cases}  
		\mu_{\tH\prec \tH'} & \text{ if } \tH\prec \tH',\  |\pi_0(\tH')| <  |\pi_0(\tH)|\\
		 {\rm I}_{\tH\prec \tH'} & \text{ if } \tH\prec \tH',\ |\pi_0(\tH')| = |\pi_0(\tH)|\\
		\mathrm{Id}_{\mathcal{F}_{A}(\tH)} & \text{ if } \tH= \tH'
	\end{cases} \ .
\end{equation}
Equations~\eqref{eq:fun_obj} and \eqref{eq:fun_mor} describe a functor 
\begin{equation}\label{eq:functor_pathposet}
\mathcal{F}_{A}\colon \mathbf{P}_\mathscr{M}(\tG)\to R\text{-}\mathbf{Mod}
\end{equation}
from the category $ \mathbf{P}_{\mathscr{M}}(\tG)$ to the additive category $R$-$\mathbf{Mod}$ of left $R$-modules. In fact, we have the following:

\begin{prop}\label{prop:functpreservessq}
The assignment $\mathcal{F}_{A}\colon \mathbf{P}_{\mathscr{M}}(\tG)\to R\text{-}\mathbf{Mod}$ defines a covariant functor.
\end{prop}

\begin{proof}
The assignment $\mathcal{F}_{A}(\tH\prec \tH')\coloneqq  \mu_{\tH\prec \tH'}$ in Equation~\eqref{eq:fun_mor} preserves all commutative squares in $\mathbf{P}_{\mathscr{M}}(\tG)$ -- here we used that the algebra $A$ is commutative (see also~\cite[Subsection~2.2]{HGRong}).
The poset $\mathbf{P}_{\mathscr{M}}(\tG)$ is a CW-poset, hence the statement follows from  \cite[Theorems~6.1 \&~5.14]{chandler2019posets}.
\end{proof}

We can summarize the results of this section in the following theorem; 
\begin{thm}
Let $\mathscr{M}$ be a monotone graph property. Then the graded $R$-module 
$C^*_{\mathcal{F}_{A}}(P_{\mathscr{M}}(\tG))$, endowed with the differential $ d^*_{\mathcal{F}_{A}}$ is a cochain complex.
\end{thm}
\begin{proof}
By Remark~\ref{rem:Boolean sings} there exists a sign assignment on $P_\mathscr{M}(\tG)$, and $\cF_{A}\colon{\bf P}_\mathscr{M}(\tG) \to \rmod$ is a functor by Proposition~\ref{prop:functpreservessq}. Therefore, by Theorem~\ref{teo: general cohom}, $(C^*_{\mathcal{F}_{A}}(P_{\mathscr{M}}(\tG)), d^*)$  is a cochain complex. 
\end{proof}

We define the associated cohomology groups:
\begin{defn}\label{def:monhom} Let $\mathscr{M}$ be a monotone property.
The \emph{monotone cohomology} $\mathrm{H}_{\mathscr{M}}^*(\tG;A)$ (with respect to $\mathscr{M}$ and $A$) of an oriented graph $\tG$ is the homology  of the cochain complex $(C^{*}_{\mathscr{M}}(\tG;A),d^*)$, where $C^{*}_{\mathscr{M}}(\tG;A)\coloneqq C^*_{\mathcal{F}_{A}}(P_{\mathscr{M}}(\tG))$. 
\end{defn}

Consider the posets of spanning subgraphs and the path poset (\emph{cf.}~Examples~\ref{ex:spanning} and~\ref{ex:multipath}). The associated monotone cohomologies are the chromatic homology~\cite{HGRong} and the multipath cohomology \cite{primo}, respectively.

In the definition of $\cF_A$, one can replace $I_{\tH\prec \tH'}$ with the zero morphism. The replacement yields again a well-defined cochain complex. When $\mathscr{M} = SS(\tG)$, this was done by Przyticki~\cite{Prz} to obtain a variation of the chromatic homology. This theory was used to provide a connection between Khovanov homology and Hochschild homology.

\section{Oriented matching complexes}

The purpose of this section is  twofold; first, we introduce several classical concepts related to graphs, such as matchings and their associated simplicial complexes. Then, we show how to extend this to the oriented setting by defining oriented matchings. These matchings  turn out to be related to a special kind of orientations, called free-flow. 

\subsection{Matchings}

A (oriented or unoriented) graph~$\tG$ can be regarded as a $1$-dimensional simplicial complex. We denote by $F(\tG)$ its face poset; this is the poset whose elements are the non-empty simplices in $\tG$, and whose order is given by inclusion. 

\begin{figure}
\centering
\begin{tikzpicture}[thick]
%%%%%%%%%%%%%%%%%%
%%  Triangolo   %%
%%%%%%%%%%%%%%%%%%
\node (a) at (-1,0) {};
\node (b) at (1,0) {};
\node (c) at (0,1.73) {};

\draw[fill, black] (a) circle (.05) node[left] {$v_{0}$};
\draw[fill, black] (b) circle (.05) node[right] {$v_{1}$};
\draw[fill, black] (c) circle (.05) node[above] {$v_{2}$};

\draw[bunired] (a) -- (b);
\draw[bunired] (c) -- (b);
\draw[bunired] (a) -- (c);

%%%%%%%%%%%%%%%%%%
%%  Face Poset  %%
%%%%%%%%%%%%%%%%%%

\begin{scope}[shift = {+(5,0)}]

\node (a) at (-2,0) {};
\node (b) at (0,0) {};
\node (c) at (2,0) {};
\node (ac) at (0, 1.73) {};
\node (bc) at (2, 1.73) {};
\node (ab) at (-2, 1.73) {};

\draw[fill, black] (a) circle (.05) node[below] {$v_{0}$};
\draw[fill, black] (b) circle (.05) node[below] {$v_{1}$};
\draw[fill, black] (c) circle (.05) node[below] {$v_{2}$};

\draw[] (ac) circle (.05) node[above] {$(v_{0},v_2)$};
\draw[] (bc) circle (.05) node[above] {$(v_{1},v_2)$};
\draw[] (ab) circle (.05) node[above] {$(v_0,v_{1})$};

\draw[-latex] (ab) -- (b);
\draw[-latex] (ac) -- (c);

\draw[-latex] (ab) -- (a);
\draw[-latex] (bc) -- (c);
\draw[-latex] (ac) -- (a);

\draw[line width = 3, white] (ab) -- (b);
\draw[-latex] (ab) -- (b);

\draw[line width = 3, white] (bc) -- (b);
\draw[-latex] (bc) -- (b);
\end{scope}

%%%%%%%%%%%%%%%%%%
%% Baricentrica %%
%%%%%%%%%%%%%%%%%%

\begin{scope}[shift = {+(10,0)}]

\node (a) at (-1,0) {};
\node (b) at (1,0) {};
\node (c) at (0,1.73) {};
\node (ac) at (-.5, .866) {};
\node (bc) at (.5, .866) {};
\node (ab) at (0, 0) {};

\draw[fill, black] (a) circle (.05) node[left] {$v_{0}$};
\draw[fill, black] (b) circle (.05) node[right] {$v_{1}$};
\draw[fill, black] (c) circle (.05) node[above] {$v_{2}$};

\draw[] (ac) circle (.05) node[above left] {$(v_{0},v_2)$};
\draw[] (bc) circle (.05) node[above right] {$(v_{1},v_2)$};
\draw[] (ab) circle (.05) node[below] {$(v_0,v_{1})$};

\draw[bunired, -latex] (ab) -- (b);
\draw[bunired, -latex] (bc) -- (b);
\draw[bunired, -latex] (ac) -- (c);

\draw[bunired, -latex] (ab) -- (a);
\draw[bunired, -latex] (bc) -- (c);
\draw[bunired, -latex] (ac) -- (a);
\end{scope}
\end{tikzpicture}
\caption{From left to right: an unoriented graph, its face poset  and its (oriented) barycentric subdivision.}
\label{fig:hassegraph}
\end{figure}
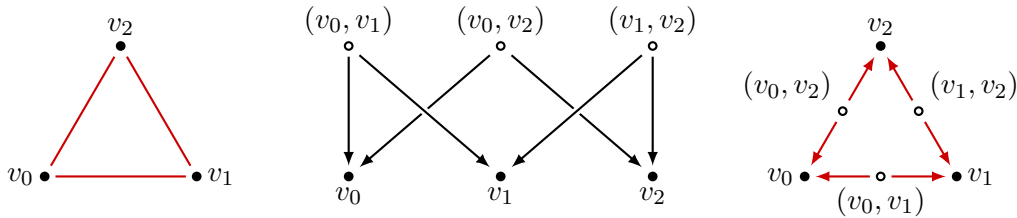

Note that $F(\tG)$ can be straightforwardly seen as an oriented graph, denoted by $\hasse(\tG)$, as follows; the vertices of $\hasse(\tG)$ are the elements in $F(\tG)$, and there is an oriented edge $(x,y)$ if $y\prec x$.
The graph  $\hasse(\tG)$  can be identified with the oriented barycentric subdivision of $\tG$, as shown in  Figure~\ref{fig:hassegraph}.

\begin{defn}\label{def:grcomplexes}
A \emph{graph matching} on a graph $\tG$ is a subset of $E(\tG)$ consisting of pairwise disjoint edges. The collection of graph matchings on $\tG$ will be denoted by $\overline{\MM}(\tG)$.
\end{defn}

\begin{defn}\label{def:complexes}
For a simplicial complex~$X$, a \emph{matching} on $X$ is a graph matching on the underlying unoriented graph of~$\hasse(X)$. 
We denote the set of all matchings on $X$ by~$\MM(X)$.
\end{defn}
Note that, given a simplicial complex $X$, we have $\MM(X) = \overline{\MM}(\hasse(X))$.
In particular, there are at least two kinds of matchings that can be considered on a simple graph $\tG$, that is $\overline{\MM}(\tG)$ and~$\MM(\tG)$. 
Observe that $\MM(X)$ and $\overline{\MM}(\tG)$ admit a natural simplicial structure; the $i$-simplices are the matchings with $i$ edges.
To stress the difference between these two, we call the former simplicial complex  \emph{matching complex}, and the latter the \emph{graph matching complex}.%

In~\cite{celoria2020filtered} a $\bN$-valued filtration $J$ on $\MM(X)$ was defined. Roughly speaking, the value of~$J$ on a matching $m \in \MM(X)$ is the number of oriented cycles in $\hasse(X)$ obtained by inverting the orientation of all the edges in $m$.
For $j \in \bN$ set $\MM_j(X) = J^{-1}([0,j])$. Each $\MM_j(X)$ is a simplicial subcomplex of $\MM(X)$. The elements of $\MM_0(X) \eqqcolon \mor (X)$ are called \emph{discrete Morse matchings}~\cite{chari2000discrete}.
 
Given an unoriented graph~$\tG$, denote by ${O}(\tG)$ the set of all possible orientations on $\tG$. For each $\mathfrak{o}\in {O}(\tG)$, the corresponding oriented graph will be denoted by $\orient$. 
It is apparent from Definition~\ref{def:complexes} that, for an oriented graph $\tG$, $\MM(\tG)$ does not depend on the orientation of the edges of $\tG$. We include this information as follows:
\begin{defn}\label{def:omatch}
An \emph{oriented matching} $m$ on an oriented graph $\tG$ is a matching on the subgraph of $\hasse(\tG)$ consisting of the edges connecting the barycentres of the edges of $\tG$ to their targets.
\end{defn}

\begin{example}
Consider the oriented graph $\tG$ depicted on the left side of Figure~\ref{fig:cycleoriented}. The matching $\{ ((v_0,v_1),v_1), ((v_0,v_2),v_2)\}$ is an oriented matching and $\{ ((v_0,v_1),v_1), ((v_2,v_1),v_1)\}$ is not.
\end{example}

We denote by $\mat(\tG)$ the simplicial complex consisting of the oriented matchings on an oriented graph $\tG$.

\begin{figure}[ht]
\centering
\begin{tikzpicture}[thick]
%%%%%%%%%%%%%%%%%%
%%  Triangolo   %%
%%%%%%%%%%%%%%%%%%
\node (a) at (-1,0) {};
\node (b) at (1,0) {};
\node (c) at (0,1.73) {};
\node (d) at (-2,1.73) {};
\node (e) at (-4,1.73) {};
\node (f) at (-3,0) {};

\draw[fill, black] (a) circle (.05) node[left] {$v_{0}$};
\draw[fill, black] (b) circle (.05) node[right] {$v_{1}$};
\draw[fill, black] (c) circle (.05) node[above] {$v_{2}$};
\draw[fill, black] (d) circle (.05) node[above] {$v_{3}$};
\draw[fill, black] (e) circle (.05) node[above] {$v_{4}$};
\draw[fill, black] (f) circle (.05) node[left] {$v_{5}$};

\draw[bunired,-latex] (a) -- (b);
\draw[bunired,-latex] (c) -- (b);
\draw[bunired,-latex] (a) -- (c);

\draw[bunired,-latex] (c) -- (d);
\draw[bunired,-latex] (f) -- (d);
\draw[bunired,-latex] (d) -- (e);

%%%%%%%%%%%%%%%%%%
%% Baricentrica %%
%%%%%%%%%%%%%%%%%%

\begin{scope}[shift = {+(7,0)}]

\node (a) at (-1,0) {};
\node (b) at (1,0) {};
\node (c) at (0,1.73) {};
\node (d) at (-2,1.73) {};
\node (e) at (-4,1.73) {};
\node (f) at (-3,0) {};

\node (ac) at (-.5, .866) {};
\node (bc) at (.5, .866) {};
\node (ab) at (0, 0) {};
\node (cd) at (-1,1.73) {};
\node (de) at (-3,1.73) {};
\node (fd) at (-2.5,0.866) {};

\draw[fill, black] (a) circle (.05) node[left] {$v_{0}$};
\draw[fill, black] (b) circle (.05) node[right] {$v_{1}$};
\draw[fill, black] (c) circle (.05) node[above] {$v_{2}$};
\draw[fill, black] (d) circle (.05) node[above] {$v_{3}$};
\draw[fill, black] (e) circle (.05) node[above] {$v_{4}$};
\draw[fill, black] (f) circle (.05) node[left] {$v_{5}$};

\draw[] (ac) circle (.05) node[above left] {};%  {$(v_{0},v_2)$};
\draw[] (bc) circle (.05) node[above right] {};% {$(v_{1},v_2)$};
\draw[] (ab) circle (.05) node[below] {};% {$(v_0,v_{1})$};
\draw[] (cd) circle (.05) node[above] {};% {$(v_{0},v_2)$};
\draw[] (de) circle (.05) node[above] {};%  {$(v_{1},v_2)$};
\draw[] (fd) circle (.05) node[left]  {};%  {$(v_0,v_{1})$};

\draw[bunired, -latex] (ab) -- (b);
\draw[gray, -latex] (ab) -- (a);

\draw[bunired, -latex] (bc) -- (b);
\draw[gray, -latex] (bc) -- (c);

\draw[gray, -latex] (ac) -- (a);
\draw[bunired, -latex] (ac) -- (c);

\draw[gray, -latex] (cd) -- (c);
\draw[bunired, -latex] (cd) -- (d);

\draw[bunired, -latex] (de) -- (e);
\draw[gray, -latex] (de) -- (d);

\draw[gray, -latex] (fd) -- (f);
\draw[bunired, -latex] (fd) -- (d);
\end{scope}
\end{tikzpicture}
\caption{An oriented graph (left) and the graph associated to its face poset (right). In the latter, we only kept in red the edges which can be used to construct an oriented matching, and shaded in gray the remaining ones.}
\label{fig:cycleoriented}
\end{figure}
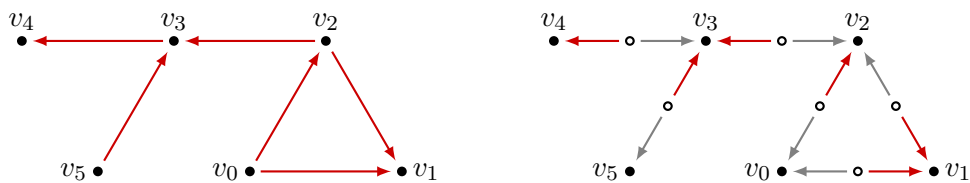

\subsection{Pseudoforests and free-flow orientations}\label{subs:pseudoforests}
The combinatorics of the simplicial complex of oriented matchings is related to certain orientations on graphs. 
We start by recalling some basic definitions.

\begin{defn}
A \emph{pseudotree} is a connected (unoriented) graph containing at most one cycle. A \emph{pseudoforest} is a disjoint union of pseudotrees. A pseudotree not containing any cycle is a \emph{tree}; a \emph{rooted tree} is a tree together with the choice of a preferred vertex called \emph{root}. We will say that a pseudoforest is \emph{rooted} if every tree component is rooted.
\end{defn}

By definition, the set $V(\tG)$ considered as a subgraph, is a spanning pseudoforest. More in general, ``being a spanning pseudoforest'' is a monotone property.

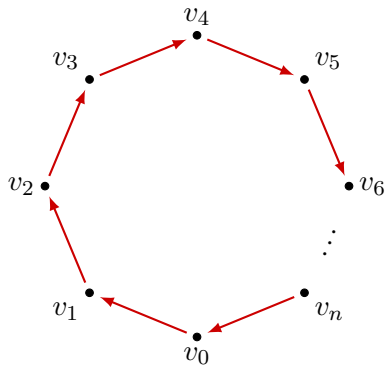
\begin{figure}[ht]
\newdimen\R
\R=2.0cm
\begin{tikzpicture}
\draw[xshift=5.0\R, fill] (270:\R) circle(.05)  node[below] {$v_0$};
\draw[xshift=5.0\R,fill] (225:\R) circle(.05)  node[below left]   {$v_1$};
\draw[xshift=5.0\R,fill] (180:\R) circle(.05)  node[left] {$v_2$};
\draw[xshift=5.0\R,fill] (135:\R) circle(.05)  node[above left] {$v_3$};
\draw[xshift=5.0\R, fill] (90:\R) circle(.05)  node[above] {$v_4$};
\draw[xshift=5.0\R,fill] (45:\R) circle(.05)  node[above right] {$v_5$};
\draw[xshift=5.0\R,fill] (0:\R) circle(.05)  node[right] {$v_6$};
\draw[xshift=5.0\R,fill] (315:\R) circle(.05)  node[below right] {$v_{n}$};

\node[xshift=5.0\R] (v0) at (270:\R) { };
\node[xshift=5.0\R] (v1) at (225:\R) { };
\node[xshift=5.0\R] (v2) at (180:\R) { };
\node[xshift=5.0\R] (v3) at (135:\R) { };
\node[xshift=5.0\R] (v4) at (90:\R) { };
\node[xshift=5.0\R] (v5) at (45:\R) { };
\node[xshift=5.0\R] (v6) at (0:\R) { };
\node[xshift=5.0\R] (vn) at (315:\R) { };

\draw[thick, bunired, -latex] (v0)--(v1);
\draw[thick, bunired, -latex] (v1)--(v2);
\draw[thick, bunired, -latex] (v2)--(v3);
\draw[thick, bunired, -latex] (v3)--(v4);
\draw[thick, bunired, -latex] (v4)--(v5);
\draw[thick, bunired, -latex] (v5)--(v6);
\draw[thick, bunired, -latex] (vn)--(v0);
\draw[xshift=4.95\R,fill] (337.5:\R)  node {$\cdot$} ;
\draw[xshift=4.95\R,fill] (333:\R)  node {$\cdot$} ;
\draw[xshift=4.95\R,fill] (342:\R)  node {$\cdot$} ;
\end{tikzpicture}
\caption{The coherently oriented cycle $\tP_n$. }
\label{fig:poly}
\end{figure}

\begin{defn}\label{def:freeflow}
We say that an orientation $\mathfrak{o}$ on a rooted tree is \emph{free-flow} if the edges of the tree are all oriented away from the root. If $\tG$ is a pseudotree which is not a tree, $\mathfrak{o}$ is free-flow if the unique cycle of $\tG$ is a coherently oriented cycle (see Figure~\ref{fig:poly}), and all the remaining edges are oriented away from the cycle (see Figure~\ref{fig:freeflow}).
\end{defn}

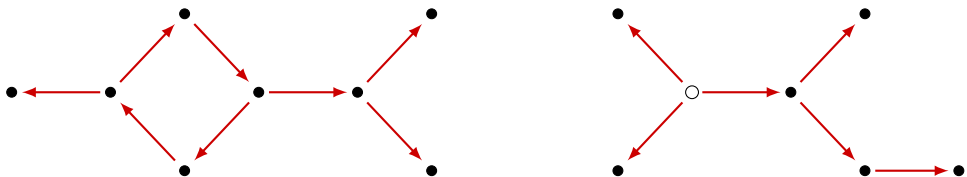
\begin{figure}
    \centering
    \begin{subfigure}{0.4\textwidth}
     \centering
    \begin{tikzpicture}[scale= 1.3]
     \node (a) at (0,0) {};
     \node (b) at (1.5,0) {};
     \node (c) at (.75,.8) {};
     \node (d) at (-1,0) {};
     \node (e) at (.75,-.8) {};
     \node (f) at (2.5,0) {};
     \node (g) at (3.25,.8) {};
     \node (h) at (3.25,-.8) {};

     \draw[fill] (a) circle (.05);
     \draw[fill] (b) circle (.05);
     \draw[fill] (c) circle (.05);
     \draw[fill] (d) circle (.05);
     \draw[fill] (e) circle (.05);
     \draw[fill] (f) circle (.05);
     \draw[fill] (g) circle (.05);
     \draw[fill] (h) circle (.05);
     
     \draw[thick, -latex, bunired] (a) -- (c);
     \draw[thick, -latex, bunired] (c) -- (b);
     \draw[thick, -latex, bunired] (e) -- (a);
     \draw[thick, -latex, bunired] (a) -- (d);
     \draw[thick, -latex, bunired] (b) -- (e);
     \draw[thick, -latex, bunired] (b) -- (f);
     \draw[thick, -latex, bunired] (f) -- (g);
     \draw[thick, -latex, bunired] (f) -- (h);
    \end{tikzpicture}
    \end{subfigure}
    \hspace{0.05\textwidth}
    \begin{subfigure}{0.4\textwidth}
     \centering
    \begin{tikzpicture}[scale= 1.3]
     \node (b) at (1.5,0) {};
     \node (c) at (.75,.8) {};
     \node (d) at (4.2,-.8) {};
     \node (e) at (.75,-.8) {};
     \node (f) at (2.5,0) {};
     \node (g) at (3.25,.8) {};
     \node (h) at (3.25,-.8) {};

     \draw[] (b) circle (.065);
     \draw[fill] (c) circle (.05);
     \draw[fill] (d) circle (.05);
     \draw[fill] (e) circle (.05);
     \draw[fill] (f) circle (.05);
     \draw[fill] (g) circle (.05);
     \draw[fill] (h) circle (.05);

     \draw[thick, -latex, bunired] (b) -- (c);
     \draw[thick, -latex, bunired] (h) -- (d);
     \draw[thick, -latex, bunired] (b) -- (e);
     \draw[thick, -latex, bunired] (b) -- (f);
     \draw[thick, -latex, bunired] (f) -- (g);
     \draw[thick, -latex, bunired] (f) -- (h);  \end{tikzpicture}
    \end{subfigure}
    \caption{Free-flow orientations on a pseudo-tree and a tree. The root of the tree is shown in white.}
    \label{fig:freeflow}
\end{figure}

In particular, it follows that a pseudotree that is not a tree has exactly two free-flow orientations, obtained from one another by inverting the orientation on the cycle. On the other hand, a tree with $n$ vertices has exactly $n$ free-flow orientations, each determined by the choice of a vertex acting as the unique root.

For an oriented graph $\tG$, let $v\in V(\tG)$ be a vertex. The \emph{indegree} (respectively \emph{outdegree}) \emph{of $v$}, denoted by $\mathrm{indeg}(v)$, is the number of edges of $\tG$ whose target (respectively source) is $v$.
It is straightforward to obtain the following characterisation of the connected components of a free-flow pseudoforest, which will be used throughout the rest of this section. 
\begin{lem}\label{lem:pseudoindegree}
Let $\tG$ be an oriented graph. Then $\tG$ is a free-flow pseudoforest if and only if the indegree of each vertex is either $0$ or $1$.  
\end{lem}
\begin{proof}
One direction is immediate: if $\tG$ is a free-flow pseudoforest, then by definition the indegree of its vertices is less or equal to one. 
Conversely, for $\tG'$ a connected component of $\tG$, let $V = |V(\tG')|$ and $E=|E(\tG')| = \sum_{v\in V(\tG')} \mathrm{indeg}(v)$. By assumption $\chi(\tG') = V - E \ge 0$, but since~$\tG'$ is connected (hence in particular homotopic to a wedge of circles), we have that $\chi(\tG') \le 1$. Therefore, as an unoriented graph, $\tG'$ is either a tree or a pseudotree.  It is now easy to see that the requirement $\mathrm{indeg}(v) \le 1$ for all vertices implies that the orientation is free-flow.
\end{proof}

\begin{prop}\label{prop:simplessi_psuedoforeste}
Let $\tG$ be an oriented graph.
Simplices in $\mat(\tG)$ are in bijection with rooted spanning pseudoforests in $\tG$ whose induced orientation is free-flow. In particular, $\mat(\tG)$ consists of a unique simplex (of dimension $|E(\tG)|-1$) if and only if $\tG$ is a free-flow pseudotree.
\end{prop}
\begin{proof}
Let $F$ be a rooted pseudoforest in $\tG$; consider the  edges $E_F$ in ${\tt F}(\tG)$ obtained by connecting the barycentres of the edges in $F$ to their targets.
Lemma~\ref{lem:pseudoindegree} can be used to show that $E_F$ defines a matching on $\tG$, thus a simplex in $\sigma_F \subseteq \mat(\tG)$.

\begin{figure}[ht]
\centering
\begin{tikzpicture}[thick, scale = 0.9]

%%%%%%%%%%%%%%%%%%
%%  Triangolo   %%
%%%%%%%%%%%%%%%%%%
\begin{scope}[shift = {+(-12,0)}]
\node (a) at (-1,0) {};
\node (b) at (1,0) {};
\node (c) at (0,1.73) {};
\node (d) at (-2,1.73) {};
\node (e) at (-4,1.73) {};
\node (f) at (-3,0) {};
\node (h) at (0,-1.73) {};
\node (i) at (-2,-1.73) {};

\draw[fill, black] (a) circle (.05);% node[left] {$v_{0}$};
\draw[fill, black] (b) circle (.05);% node[right] {$v_{1}$};
\draw[fill, black] (c) circle (.05);% node[above] {$v_{2}$};
\draw[fill, black] (d) circle (.05);% node[above] {$v_{3}$};
\draw[fill, black] (e) circle (.05);% node[above] {$v_{4}$};
\draw[fill, black] (f) circle (.05);% node[left] {$v_{5}$};
\draw[fill, black] (h) circle (.05);% node[above] {$v_{6}$};
\draw[fill, black] (i) circle (.05);% node[left] {$v_{7}$};

\draw[bunired, -latex] (b) -- (a);
\draw[bunired, -latex] (c) -- (b);
\draw[bunired, -latex] (a) -- (c);

\draw[bunired, -latex] (a) -- (d);
\draw[bunired, -latex] (f) -- (a);
\draw[bunired, -latex] (d) -- (c);
\draw[bunired, -latex] (d) -- (f);
\draw[bunired, -latex] (f) -- (e);

\draw[bunired, -latex] (h) -- (a);
\draw[bunired, -latex] (h) -- (i);
\draw[bunired, -latex] (a) -- (i);
\end{scope}
%%%%%%%%%%%%%%%%%%
%%  Triangolo   %%
%%%%%%%%%%%%%%%%%%
\node (a) at (-1,0) {};
\node (b) at (1,0) {};
\node (c) at (0,1.73) {};
\node (d) at (-2,1.73) {};
\node (e) at (-4,1.73) {};
\node (f) at (-3,0) {};
\node (h) at (0,-1.73) {};
\node (i) at (-2,-1.73) {};

\draw[fill, black] (a) circle (.05);% node[left] {$v_{0}$};
\draw[fill, black] (b) circle (.05);% node[right] {$v_{1}$};
\draw[fill, black] (c) circle (.05);% node[above] {$v_{2}$};
\draw[fill, black] (d) circle (.05);% node[above] {$v_{3}$};
\draw[fill, black] (e) circle (.05);% node[above] {$v_{4}$};
\draw[fill, black] (f) circle (.05);% node[left] {$v_{5}$};
\draw (h) circle (.1);% node[above] {$v_{6}$};
\draw[fill, black] (i) circle (.05);% node[left] {$v_{7}$};

\draw[very thick, bunired,-latex] (b) -- (a);
\draw[very thick, bunired,-latex] (c) -- (b);
\draw[very thick, bunired,-latex] (a) -- (c);

\draw[very thick, bunired,-latex] (a) -- (d);
\draw[gray, opacity = .75, -latex] (f) -- (a);
\draw[gray, opacity = .75, -latex] (d) -- (c);
\draw[very thick, bunired,-latex] (d) -- (f);
\draw[very thick, bunired,-latex] (f) -- (e);

\draw[gray, opacity = .75, -latex] (h) -- (a);
\draw[very thick, bunired,-latex] (h) -- (i);
\draw[gray, opacity = .75, -latex] (a) -- (i);

%%%%%%%%%%%%%%%%%%
%% Baricentrica %%
%%%%%%%%%%%%%%%%%%

\begin{scope}[shift = {+(-6,0)}]

\node (a) at (-1,0) {};
\node (b) at (1,0) {};
\node (c) at (0,1.73) {};
\node (d) at (-2,1.73) {};
\node (e) at (-4,1.73) {};
\node (f) at (-3,0) {};
\node (h) at (0,-1.73) {};
\node (i) at (-2,-1.73) {};

\node (ab) at (0, 0) {};
\node (bc) at (.5, .866) {};
\node (ac) at (-.5, .866) {};
\node (cd) at (-1,1.73) {};
\node (fe) at (-3.5,0.866) {};
\node (fd) at (-2.5,0.866) {};
\node (ad) at (-1.5,0.866) {};
\node (hi) at (-1,-1.73) {};
\node (ai) at (-1.5,-0.866) {};
\node (ah) at (-.5, -0.866) {};
\node (af) at (-2, 0) {};

\draw[fill, black] (a) circle (.05);% node[left] {$v_{0}$};
\draw[fill, black] (b) circle (.05);% node[right] {$v_{1}$};
\draw[fill, black] (c) circle (.05);% node[above] {$v_{2}$};
\draw[fill, black] (d) circle (.05);% node[above] {$v_{3}$};
\draw[fill, black] (e) circle (.05);% node[above] {$v_{4}$};
\draw[fill, black] (f) circle (.05);% node[left] {$v_{5}$};
\draw[fill, black] (h) circle (.05);% node[above] {$v_{6}$};
\draw[fill, black] (i) circle (.05);% node[left] {$v_{7}$};

\draw[] (ac) circle (.05) node[above left] {};%  {$(v_{0},v_2)$};
\draw[] (bc) circle (.05) node[above right] {};% {$(v_{1},v_2)$};
\draw[] (ab) circle (.05) node[below] {};% {$(v_0,v_{1})$};
\draw[] (ad) circle (.05) node[above] {};% {$(v_{0},v_2)$};
\draw[] (cd) circle (.05) node[above] {};% {$(v_{0},v_2)$};
\draw[] (fe) circle (.05) node[above] {};%  {$(v_{1},v_2)$};
\draw[] (fd) circle (.05) node[left]  {};%  {$(v_0,v_{1})$};
\draw[] (hi) circle (.05) node[left]  {};%  {$(v_0,v_{1})$};
\draw[] (ah) circle (.05) node[below] {};% {$(v_0,v_{1})$};
\draw[] (ai) circle (.05) node[above] {};% {$(v_{0},v_2)$};
\draw[] (af) circle (.05) node[above] {};% {$(v_{0},v_2)$};

\draw[bunired, -latex] (ab) -- (a);
\draw[gray, -latex] (ab) -- (b);

\draw[bunired, -latex] (ad) -- (d);
\draw[gray, -latex] (ad) -- (a);

\draw[bunired, -latex] (bc) -- (b);
\draw[gray, -latex] (bc) -- (c);

\draw[gray, -latex] (ac) -- (a);
\draw[bunired, -latex] (ac) -- (c);

\draw[gray, -latex] (cd) -- (c);
\draw[gray, -latex] (cd) -- (d);

\draw[bunired, -latex] (fe) -- (e);
\draw[gray, -latex] (fe) -- (f);

\draw[bunired,  -latex] (fd) -- (f);
\draw[gray,-latex] (fd) -- (d);

\draw[bunired,  -latex] (hi) -- (i);
\draw[gray,-latex] (hi) -- (h);

\draw[gray, -latex] (ai) -- (i);
\draw[gray, -latex] (ai) -- (a);

\draw[gray, -latex] (ah) -- (h);
\draw[gray, -latex] (ah) -- (a);

\draw[gray, -latex] (af) -- (f);
\draw[gray, -latex] (af) -- (a);
\end{scope}
\end{tikzpicture}
\caption{From left to right: an oriented graph $\tG$, $\hasse(\tG)$ with a (maximal) oriented matching $m\in \mat(\tG)$ highlighted in red, and the rooted free-flow pseudoforest induced by $m$ on $\tG$.}
\label{fig:matchingandpseudoforest}
\end{figure}
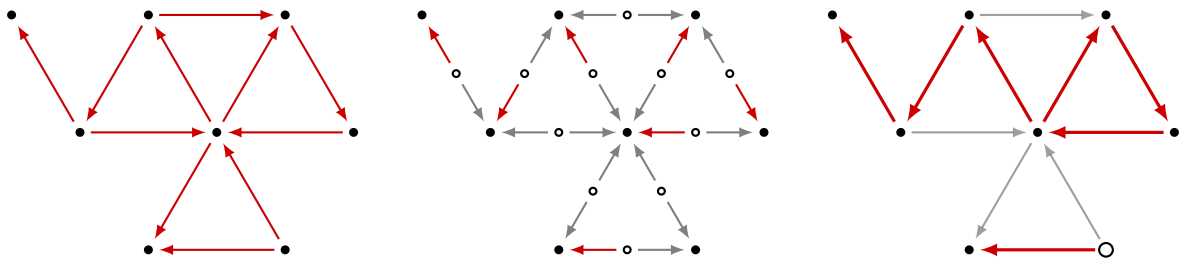

For the converse, let $\sigma \subseteq \mat(\tG)$ be any simplex. We want to show that $\sigma$ uniquely determines a pseudoforest $F_\sigma \subseteq \tG$ (see Figure~\ref{fig:matchingandpseudoforest}), and that moreover the orientation of $F_\sigma$ is free-flow.  Each vertex in $\sigma$ uniquely determines a directed edge in $\tG$, and we take $F_\sigma$ to be the union of all of these edges. 
We can conclude by noting that, as $\sigma \in \mat(\tG)$, the indegree of each vertex in $F_\sigma$ must be $\le 1$. Therefore, by Lemma~\ref{lem:pseudoindegree} $F_\sigma$ is a free-flow pseudoforest.

The second part of the statement follows readily from the previous one.
\end{proof}

Proposition \ref{prop:simplessi_psuedoforeste} implies, in particular, that simplices in $\mor(\tG)$ are in bijection with free-flow oriented spanning forests in~$\tG$. This was the starting point of \cite{MR2179635}, where Kozlov's complex of directed forests was identified with the discrete Morse complex they defined:
\begin{prop}[{\cite[Proposition~3.1]{MR2179635}}]
The set of discrete Morse matchings on an unoriented graph~$\tG$ is in one-to-one correspondence with the set of rooted forests of $\tG$.
\end{prop}

It is easy to see that, for each orientation $\mathfrak{o} \in \mathcal{O}(\tG)$ and $j \in \bN$, $\mat_j(\orient)\coloneqq \mat(\orient) \cap \MM_j(\tG)$ is a simplicial subcomplex of $\MM_j(\tG)$; in particular this holds for the \emph{oriented discrete Morse matchings} $\mathcal{M}^o(\orient)\coloneqq \mat_0(\orient)$. 
We remark that a simple application of~\cite[Proposition~2.9]{celoria2020filtered} shows that simplices in $\mat_j(\tG)$ are in bijection with free-flow pseudoforests with at most $j$ pseudotrees that are not trees. 
Moreover, 
\begin{equation}\label{eq:matching as union}
\mathrm{M}(\tG) = \bigcup_{\mathfrak{o} \in \mathcal{O}(\tG)} \mat(\orient) \ .
\end{equation}
In other words, all matchings on $\tG$ arise as oriented matchings for some orientation on $\tG$; see Figure~\ref{fig:decomposition} for an example. Furthermore, Equation~\eqref{eq:matching as union} provides a decomposition of the matching complex $\mathrm{M}(\tG) $ in terms of the oriented ones.

\begin{figure}[ht]
\centering
\includegraphics[width=9cm]{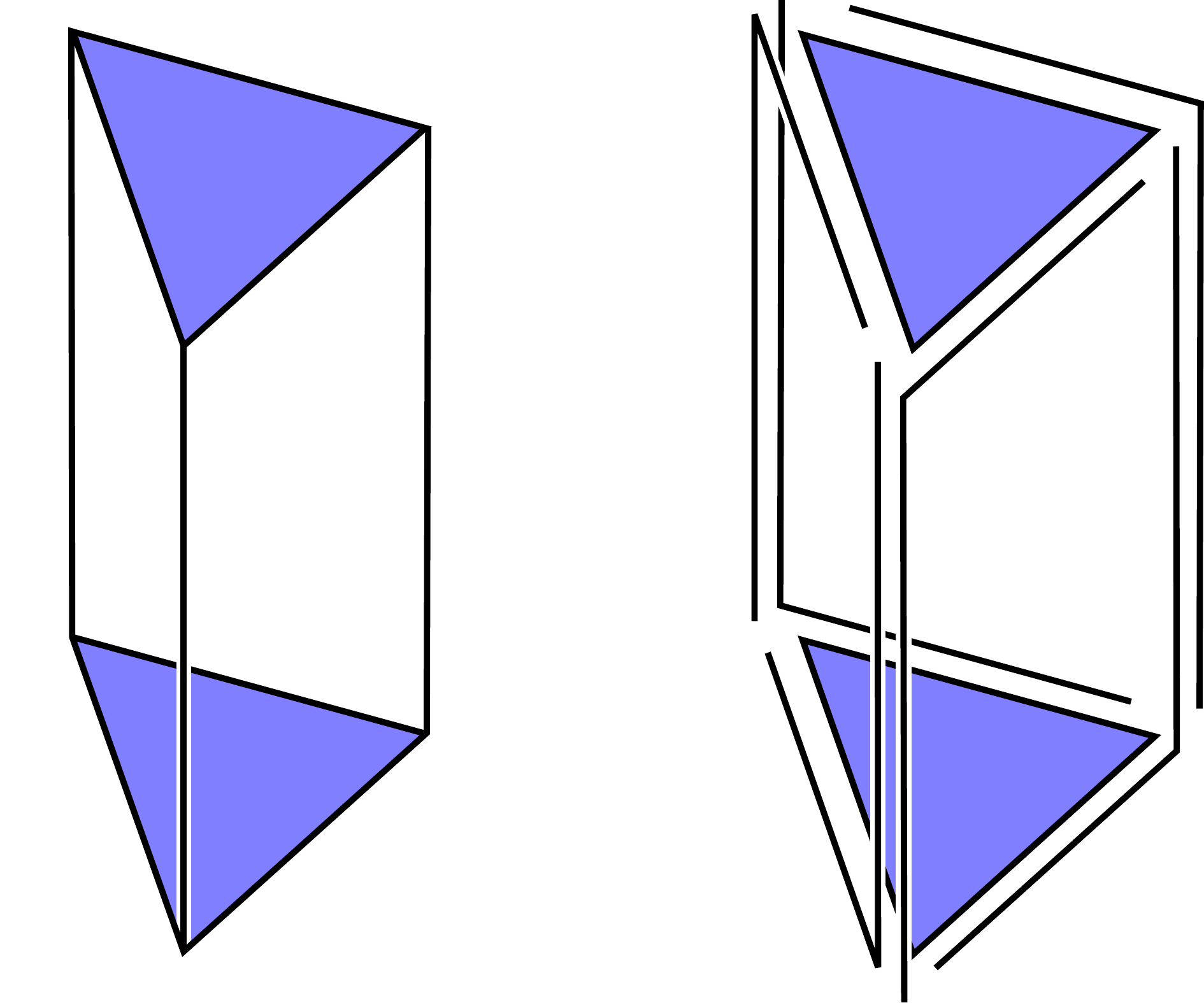}
\caption{The decomposition of the matching complex $\mathrm{M}(\tC_3)$ for the cycle graph of length three (on the left) into its $8$ pieces corresponding to the oriented matching complexes (right). The two $2$-simplices correspond to the cyclic orientations on $\tC_3$, while the other six $1$-dimensional components are induced by all the other possible orientations.}
\label{fig:decomposition}
\end{figure}

\section{Relations between matchings and multipaths}

In this section  we prove that, in certain cases, matchings on graphs and path posets  can be identified. First, we study the case of graph matching complexes on unoriented graphs. Then, we provide an isomorphism between oriented matching complexes on oriented graphs and path posets. 

Call an orientation $\oo\in O(\tG)$ on an unoriented graph $\tG$ \emph{alternating} if there exists a partition~$V\sqcup W$ of $V(\tG)$ such that all elements of  $V$ have indegree $0$ and all elements of $W$ have outdegree $0$. Note that the existence of an alternating orientation implies that $\tG$ is a bipartite graph.

Recall that, for an oriented graph $\tG$, $P(\tG)$ denotes its path poset (\emph{cf.}~Example~\ref{ex:multipath}) and that $F(S)$ denotes the face poset of a simplicial complex $S$. Then, we have the following result: 

\begin{thm}\label{thm:multi=match}
Let $\tG$ be an unoriented graph. Then, we have an isomorphism $F(\overline{\MM}(\tG)) \cong P(\orient)$ if and only if $\mathfrak{o}$ is alternating.
\end{thm}

\begin{proof}
First, note that every graph matching in $\overline{\MM}(\tG)$ can be regarded as being a multipath in~$\tG_\mathfrak{o}$, independently on the chosen orientation~$\mathfrak{o}\in O(\tG)$.
In particular, $\vert P(\orient) \vert \geq \vert F(\overline{\MM}(\tG)) \vert$.
We claim that, if we fix an arbitrary alternating orientation $\mathfrak{o}\in \mathcal{O}(\tG)$, then every multipath induces a matching on $\orient$.

Since $\oo$ is alternating, there is a partition $V\sqcup W = V(\tG)$ such that all the edges of $\orient$ are of the form $(v,w)$ with  $v\in V,$ and $w\in W$.
We only have to observe that the connected components of any multipath in $\orient$ are either vertices or single edges. For the sake of contradiction, assume there exists at least one component $c$ of a certain multipath $\tH$ that is not a single edge or a vertex. Then, there are at least two edges in $c$ which share a vertex; in particular the target of one edge, which is a vertex in $W$, must be the source of another edge, and therefore it is in $V$. This is a contradiction since $V\cap W$ is empty, and the ``if'' part of the statement follows.

For the converse, assume $\mathfrak{o}$ is not alternating. Then, there exists a vertex which is both a source and a target --
otherwise, the partition of $V(\orient)$ into indegree~$0$ and outdegree~$0$ vertices would imply that $\mathfrak{o}$ is alternating. 
Therefore, there is at least a multipath of length two which is not made of disjoint edges. Thus, $\vert P(\orient) \vert > \vert F(\overline{\MM}(\tG)) \vert$, concluding the proof.
\end{proof}

\begin{cor}
All alternating orientations on a graph have isomorphic path posets.
\end{cor}
Note however that is it possible to find two orientations (not both alternating) yielding isomorphic path posets.
As an example, consider the two orientations on the ``Y''-shaped graph shown in Figure~\ref{fig:Y}. 

\begin{figure}[ht]
    \centering
    	\begin{tikzpicture}[baseline=(current bounding box.center),scale =.7]
		\tikzstyle{point}=[circle,thick,draw=black,fill=black,inner sep=0pt,minimum width=2pt,minimum height=2pt]
		\tikzstyle{arc}=[shorten >= 8pt,shorten <= 8pt,->, thick]
		
		\node[above] (v0) at (0,0) {$v_0$};
		\draw[fill] (0,0)  circle (.05);
		\node[above] (v1) at (1.5,0) {$v_1$};
		\draw[fill] (1.5,0)  circle (.05);
		\node[above] (v2) at (3,1) {$v_{2}$};
		\draw[fill] (3,1)  circle (.05);
		\node[above] (v3) at (3,-1) {$v_{3}$};
		\draw[fill] (3,-1)  circle (.05);
		
		\draw[thick, bunired, -latex] (0.15,0) -- (1.35,0);
		\draw[thick, bunired, -latex] (1.65,0.05) -- (2.85,0.95);
		\draw[thick, bunired, -latex] (1.65,-0.05) -- (2.85,-0.95);
		
		\begin{scope}[shift = {+(7,0)}]
		\node[above] (v0) at (0,0) {$v_0$};
		\draw[fill] (0,0)  circle (.05);
		\node[above] (v1) at (1.5,0) {$v_1$};
		\draw[fill] (1.5,0)  circle (.05);
		\node[above] (v2) at (3,1) {$v_{2}$};
		\draw[fill] (3,1)  circle (.05);
		\node[above] (v3) at (3,-1) {$v_{3}$};
		\draw[fill] (3,-1)  circle (.05);
		
		\draw[thick, bunired, latex-] (0.15,0) -- (1.35,0);
		\draw[thick, bunired, latex-] (1.65,0.05) -- (2.85,0.95);
		\draw[thick, bunired, latex-] (1.65,-0.05) -- (2.85,-0.95);
        \end{scope}
	\end{tikzpicture}
    \caption{Two non-isomorphic Y-shaped oriented graphs with isomorphic path posets.}
    \label{fig:Y}
\end{figure}
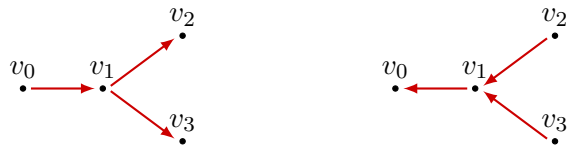

\begin{defn}
For an oriented graph $\tG$, the \emph{multipath complex} $X(\tG)$ is the simplicial complex whose face poset is the path poset $P(\tG)$. 
\end{defn}
The existence of $X(\tG)$ is guaranteed by Remark~\ref{rem:realizability} -- see also~\cite[Definition~6.4]{secondo}. 

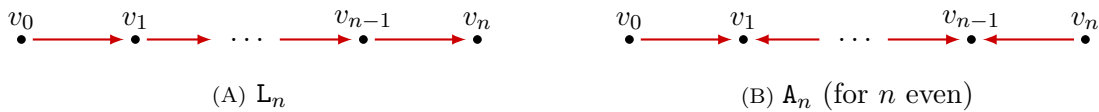
\begin{figure}[ht]
	\begin{tikzpicture}[baseline=(current bounding box.center)]
		\tikzstyle{point}=[circle,thick,draw=black,fill=black,inner sep=0pt,minimum width=2pt,minimum height=2pt]
		\tikzstyle{arc}=[shorten >= 8pt,shorten <= 8pt,->, thick]
		
		\node[above] (v0) at (0,0) {$v_0$};
		\draw[fill] (0,0)  circle (.05);
		\node[above] (v1) at (1.5,0) {$v_1$};
		\draw[fill] (1.5,0)  circle (.05);
		\node[] at (3,0) {\dots};
		\node[above] (v4) at (4.5,0) {$v_{n-1}$};
		\draw[fill] (4.5,0)  circle (.05);
		\node[above] (v5) at (6,0) {$v_{n}$};
		\draw[fill] (6,0)  circle (.05);
		
		\draw[thick, bunired, -latex] (0.15,0) -- (1.35,0);
		\draw[thick, bunired, -latex] (1.65,0) -- (2.5,0);
		\draw[thick, bunired, -latex] (3.4,0) -- (4.35,0);
		\draw[thick, bunired, -latex] (4.65,0) -- (5.85,0);
		
		\node at (3,-.75) {{\scriptsize{(A)}} $\tL_n$};

		\begin{scope}[shift = {+(8,0)}]
		    	\node[above] (v0) at (0,0) {$v_0$};
		\draw[fill] (0,0)  circle (.05);
		\node[above] (v1) at (1.5,0) {$v_1$};
		\draw[fill] (1.5,0)  circle (.05);
		\node[] at (3,0) {\dots};
		\node[above] (v4) at (4.5,0) {$v_{n-1}$};
		\draw[fill] (4.5,0)  circle (.05);
		\node[above] (v5) at (6,0) {$v_{n}$};
		\draw[fill] (6,0)  circle (.05);
		
		\draw[thick, bunired, -latex] (0.15,0) -- (1.35,0);
		\draw[thick, bunired, latex-] (1.65,0) -- (2.5,0);
		\draw[thick, bunired, -latex] (3.4,0) -- (4.35,0);
		\draw[thick, bunired, latex-] (4.65,0) -- (5.85,0);
		
		\node at (3,-.75) {{\scriptsize{(B)} }$\tA_n$ (for $n$ even)};
		\end{scope}
	\end{tikzpicture}
	\caption{{\scriptsize{(A)}} The coherently oriented linear graph $\tL_n$, and {\scriptsize{(B)}} the alternating linear graph $\tA_n$.}
	\label{fig:nstep}
\end{figure}

\begin{example}
The multipath complex $X(\tA_n)$ of the alternating graph on $n$ edges -- see Figure~\ref{fig:nstep} -- is isomorphic to the graph matching complex $\overline{\MM}(\tA_n)$, associated to the unoriented linear graph  underlying $\tA_n$.
This is coherent with the computations in \cite[Proposition~4.6]{KozlovTrees} and in \cite[Corollary~5.7]{secondo}.
\end{example}

The study of the topology of multipath complexes was initiated in \cite[Section~6]{secondo}. All examples provided therein are wedges of spheres.
A consequence of Theorem~\ref{thm:multi=match} implies that this is not always the case. The following proposition provides an affirmative answer to~\cite[Question~6.18]{secondo}.

\begin{prop}
Multipath cohomology with coefficients in $R = \bZ$, for $A=R$, can have torsion. In particular, the multipath complex is not always homotopic to a wedge of spheres.
\end{prop}
\begin{proof}
Denote by~${\tt K}_{n,m}$ the complete bipartite graph on $m,n$ vertices.
By \cite[Theorem~1.7]{torsion}, the integral homology of the graph matching complex of $\overline{\MM}({\tt K}_{n,m})$ can have torsion for certain values of $m$ and $n$ (the minimal example being $m=n=5$, containing $3$-torsion).  Therefore, $\overline{\MM}({\tt K}_{n,m})$ is not necessarily homotopic to a wedge of spheres. Theorem~\ref{thm:multi=match} implies that, for any alternating orientation $\oo$ on ${\tt K}_{n,m}$, we have $$X((\tK_{m,n})_{\oo})\cong \overline{\MM}({\tt K}_{n,m}) \ .$$ 
The multipath homology for $A = R$ is (up to a shift by $1$ in homological degree)  the reduced simplicial cohomology of the multipath complex -- \cite[Theorem~6.8]{secondo}.
The statement follows from the above isomorphism.
\end{proof}

We now want to interpret the poset of oriented matchings on a oriented graph as the path poset of a suitable oriented graph.

\begin{defn}\label{def:G_sr}
Let $\tG$ be an oriented graph, its \emph{source resolution} is the oriented graph $\tG_{sr}$ whose vertices are:
\[ V(\tG_{sr}) \coloneqq \{ v\in V(\tG) \mid v={\tt t}(e) \text{ for some } e \in E(\tG)  \} \cup \{ (v,e)\in V(\tG) \times E(\tG) \mid {\tt s}(e) = v\},\]
and whose edges are
\[ E(\tG_{sr}) \coloneqq \{ \left( (v,e) , w\right) \mid v,w \in V(\tG),\, e\in E(\tG)\text{ and } {\tt s}(e) = v, {\tt t}(e) = w  \}.  \]
Intuitively, we are splitting the sources of the edges in $\tG$ (see the top-right part of Figure~\ref{fig:sourceres}).
\end{defn}

\begin{figure}
    \centering
    \begin{subfigure}{0.4\textwidth}
     \centering
    \begin{tikzpicture}[scale= 1.2]
     \node (a) at (0,0) {};
     \node (b) at (1,0) {};
     \node (c) at (.5,.866) {};
     
     \node[below left] at (0,0) {$v_1$};
     \node[below right] at (1,0) {$v_2$};
     \node[above] at (.5,.866) {$v_3$};
     
     \draw[fill] (a) circle (.05);
     \draw[fill] (b) circle (.05);
     \draw[fill] (c) circle (.05);
     
     \draw[thick, -latex, bunired] (a) -- (b);
     \draw[thick, -latex, bunired] (b) -- (c);
     \draw[thick, -latex, bunired] (a) -- (c);
    \end{tikzpicture}
    \subcaption{$\tG$} \label{A}
    \end{subfigure}
    \hspace{0.05\textwidth}
    \begin{subfigure}{0.4\textwidth}
     \centering
    \begin{tikzpicture}[scale= 1.3]
     \node (a) at (0,-.5) {};
     \node (b) at (1,-.5) {};
     \node (c) at (.5,.866) {};
     \node (d) at (0,0) {};
     \node (e) at (1,0) {};

     \node[left] at (0,0) {$(v_1, (v_1,v_3))$};
     \node[right] at (1,0) {$(v_2,(v_2,v_3))$};
     \node[above] at (.5,.866) {$v_3$};
     \node[left] at (0,-.5) {$(v_1, (v_1,v_2))$};
     \node[right] at (1,-.5) {$v_2$};
     
     \draw[fill] (a) circle (.05);
     \draw[fill] (b) circle (.05);
     \draw[fill] (c) circle (.05);
     \draw[fill] (e) circle (.05);
     \draw[fill] (d) circle (.05);

     \draw[thick, -latex, bunired] (a) -- (b);
     \draw[thick, -latex, bunired] (e) -- (c);
     \draw[thick, -latex, bunired] (d) -- (c);
    \end{tikzpicture}
    \subcaption{$\tG_{sr}$}
    \end{subfigure}
    ~\\
    \vspace{.5em}
    \begin{subfigure}{0.4\textwidth}
     \centering
    \begin{tikzpicture}[scale= 1.2]
    
    \begin{scope}[shift ={+(-1,3)}, scale = .5]
     \node (a) at (0,-.5) {};
     \node (b) at (1,-.5) {};
     \node (c) at (.5,.866) {};
     \node (d) at (0,0) {};
     \node (e) at (1,0) {};

     \draw[fill] (a) circle (.05);
     \draw[fill] (b) circle (.05);
     \draw[fill] (c) circle (.05);
     \draw[fill] (e) circle (.05);
     \draw[fill] (d) circle (.05);
     
     \draw[thick, -latex, bunired] (a) -- (b);
     \draw[thick, -latex, bunired] (e) -- (c);
     \draw[thick, -latex, gray, opacity =.5] (d) -- (c);
     \node (H_bc) at (-1,0) {};
    \end{scope}
    \begin{scope}[shift ={+(1,3)}, scale = .5]
     \node (a) at (0,-.5) {};
     \node (b) at (1,-.5) {};
     \node (c) at (.5,.866) {};
     \node (d) at (0,0) {};
     \node (e) at (1,0) {};

     \draw[fill] (a) circle (.05);
     \draw[fill] (b) circle (.05);
     \draw[fill] (c) circle (.05);
     \draw[fill] (e) circle (.05);
     \draw[fill] (d) circle (.05);
     
     \draw[thick, -latex, bunired] (a) -- (b);
     \draw[thick, -latex, gray, opacity =.5] (e) -- (c);
     \draw[thick, -latex, bunired] (d) -- (c);
     \node (H_ac) at (-1,0) {};
    \end{scope}
        \begin{scope}[shift ={+(0,1)}, scale = .5]
     \node (a) at (0,-.5) {};
     \node (b) at (1,-.5) {};
     \node (c) at (.5,.866) {};
     \node (d) at (0,0) {};
     \node (e) at (1,0) {};

     \draw[fill] (a) circle (.05);
     \draw[fill] (b) circle (.05);
     \draw[fill] (c) circle (.05);
     \draw[fill] (e) circle (.05);
     \draw[fill] (d) circle (.05);
     
     \draw[thick, -latex, bunired] (a) -- (b);
     \draw[thick, -latex, gray, opacity =.5] (e) -- (c);
     \draw[thick, -latex, gray, opacity =.5] (d) -- (c);
    \end{scope}
    \draw[very thick] (-0.1,1.5) -- (-.65,2.5);
    \draw[very thick] (0.55,1.5) -- (1.1,2.5);
    
    \draw[very thick] (-1.45,1.5) -- (-.9,2.5);
    \draw[very thick] (1.9,1.5) -- (1.35,2.5);
    
    \draw[very thick] (-0.1,-.5) -- (-1.45,.35);
    \draw[very thick] (0.55,-.5) -- (1.9,.35);
    
    \draw[very thick] (0.25,.5) -- (0.25,-.25);
     \begin{scope}[shift ={+(-2,1)}, scale = .5]
     \node (a) at (0,-.5) {};
     \node (b) at (1,-.5) {};
     \node (c) at (.5,.866) {};
     \node (d) at (0,0) {};
     \node (e) at (1,0) {};

     \draw[fill] (a) circle (.05);
     \draw[fill] (b) circle (.05);
     \draw[fill] (c) circle (.05);
     \draw[fill] (e) circle (.05);
     \draw[fill] (d) circle (.05);
     
     \draw[thick, -latex, gray, opacity =.5] (a) -- (b);
     \draw[thick, -latex, bunired] (e) -- (c);
     \draw[thick, -latex, gray, opacity =.5] (d) -- (c);
    \end{scope}
      \begin{scope}[shift ={+(2,1)}, scale = .5]
     \node (a) at (0,-.5) {};
     \node (b) at (1,-.5) {};
     \node (c) at (.5,.866) {};
     \node (d) at (0,0) {};
     \node (e) at (1,0) {};

     \draw[fill] (a) circle (.05);
     \draw[fill] (b) circle (.05);
     \draw[fill] (c) circle (.05);
     \draw[fill] (e) circle (.05);
     \draw[fill] (d) circle (.05);

     \draw[thick, -latex, gray, opacity =.5] (a) -- (b);
     \draw[thick, -latex, gray, opacity =.5] (e) -- (c);
     \draw[thick, -latex, bunired] (d) -- (c);
    \end{scope}
      \begin{scope}[shift ={+(0,-1)}, scale = .5]
     \node (a) at (0,-.5) {};
     \node (b) at (1,-.5) {};
     \node (c) at (.5,.866) {};
     \node (d) at (0,0) {};
     \node (e) at (1,0) {};

     \draw[fill] (a) circle (.05);
     \draw[fill] (b) circle (.05);
     \draw[fill] (c) circle (.05);
     \draw[fill] (e) circle (.05);
     \draw[fill] (d) circle (.05);
 
     \draw[thick, -latex, gray, opacity =.5] (a) -- (b);
     \draw[thick, -latex, gray, opacity =.5] (e) -- (c);
     \draw[thick, -latex, gray, opacity =.5] (d) -- (c);
    \end{scope}
    \end{tikzpicture}
    \subcaption{The path poset of $\tG_{sr}$.}
    \end{subfigure}
    \caption{{\scriptsize{(A)}} A graph, {\scriptsize{(B)}} its source resolution and {\scriptsize{(C)}} the path poset of its source resolution. }
    \label{fig:sourceres}
\end{figure}
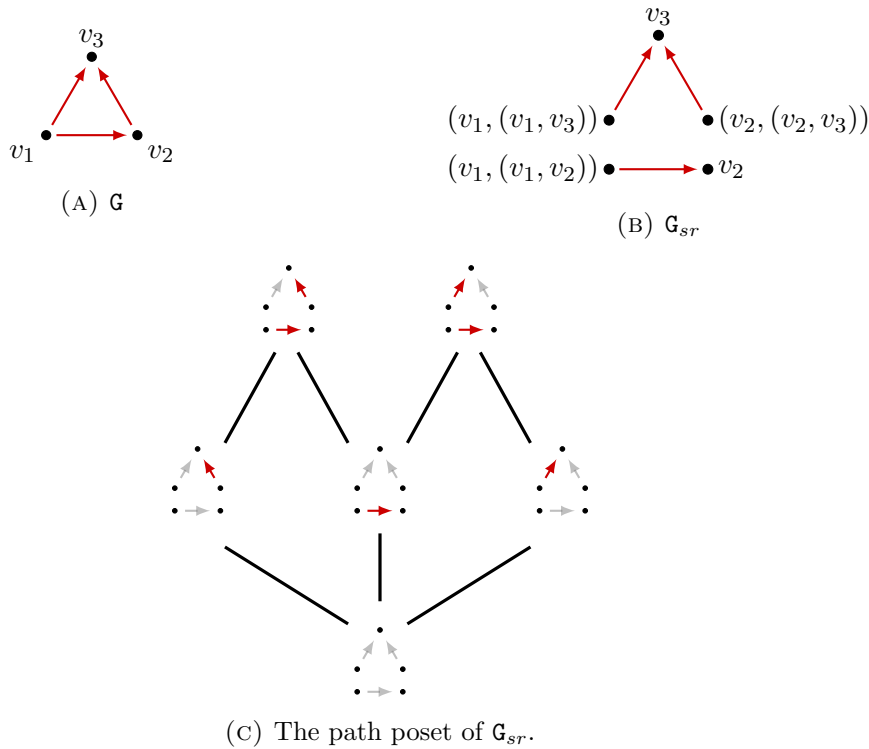
\begin{rem}\label{rem:multipaths_in_sr}
Given a graph $\tG$, each connected component of $\tG_{sr}$ is a sink. Therefore, multipaths in $\tG_{sr}$ are collections of edges whose targets are pairwise distinct.
\end{rem}

It is not hard to see that the association $\tG \mapsto P(\tG)$ is functorial (\emph{cf}.~\cite[Remark~2.33]{primo}).
This association can be promoted to a functor $\tG\mapsto P(\tG_{sr})$, with respect to regular graph morphisms. 
To this end, we have to prove that a regular morphism of oriented graphs $f\colon \tG\to \tG'$ induces a regular morphism $f_{sr}\colon  \tG_{sr}\to\tG'_{sr}$. 
Recall that a vertex $v_{sr}$ in $\tG_{sr}$ is either of the form $v\in V(\tG)$ or of the form $(v,e)$, where $v$ is the source of the edge $e\in E(\tG)$. Define
\[ f_{sr} :  V(\tG_{sr}) \to V(\tG'_{sr})\;\;\;\; v_{sr} \mapsto \begin{cases}
f(v) & \text{if }v_{sr} = v \in V(\tG);\\
(f(v), f(e)) & \text{if } v_{sr} = (v,e) \in V(\tG)\times E(\tG).
\end{cases}\] 
The definition of $f_{sr}$ is well-posed; if $v\in V(\tG_{sr})$, then $v$ is the target of at least one edge in~$\tG$. Since $f$ is regular, then $f(v)$ is the target of at least one edge in $\tG'$. As a consequence, $f(v)$ is a vertex of $\tG'_{sr}$. Similarly, one can show that if $(v,e)\in V(\tG_{sr})$, then $(f(v),f(e)) \in V(\tG'_{sr})$.
It follows directly from the definitions that $f_{sr}$ is a regular morphism of oriented graphs, and that $f\mapsto f_{sr}$ preserves compositions, proving the desired functoriality.

\begin{prop}\label{prop:or_mat=multi_sr}
Let $\tG$ be an oriented graph.
There is an isomorphism  $F(\MM^o(\tG)) \cong P(\tG_{sr})$  between the face poset of $\MM^o(\tG)$ and the path poset of $\tG_{sr}$.
\end{prop}

\begin{proof}
The map
\begin{align*}
\Phi\colon \,& E(\tG) \longrightarrow E(\tG_{sr})\\ &e \longmapsto (({\tt s}(e),e),{\tt t}(e)),
\end{align*}
is a bijection, with inverse $((v,e),w) \mapsto e$.
Therefore, we have a bijection between the (coherently oriented) edges in the barycentric  subdivision of $\tG$ and the edges of $\tG_{sr}$. 

We claim that this bijection sends an oriented matching in $\tG$ to the set of edges in a multipath in $\tG_{sr}$. An oriented matching is a collection of coherently oriented edges in the barycentric subdivision of $\tG$ not sharing the same target. These can be seen as a collection of edges in $\tG$ not sharing the same target. Since ${\tt t}(e)\neq {\tt t}(e')$ if and only if ${\tt t}(\Phi(e))\neq {\tt t}(\Phi(e'))$, our claim is a consequence of Remark~\ref{rem:multipaths_in_sr}.
A similar reasoning, replacing $\Phi$ with its inverse, can be used to prove that all oriented matchings arise in this way. 

Finally, the bijection between $\MM^o(\tG)$ and $P(\tG_{sr})$ just described clearly respects inclusions, and the statement follows.
\end{proof}

\section{Oriented matchings and monotone cohomology}
The aim of this section is to study the homotopy type of the oriented matching complexes. We show that the oriented matching complexes are homotopy equivalent to wedges of spheres; futher, the sphere's dimensions are related to simple combinatorial information of graphs. By providing an isomorphism with multipath cohomology, we show that computations can be carried on also in the case of non-constant functor coefficients.

\subsection{Homotopy type of oriented matchings}\label{sec:homotopy}

Recall that if $X$ is a finite simplicial complex, then its \emph{$n$-point suspension}\footnote{This operation is also sometimes referred to as \emph{iterated suspension}.}   $\Sigma_n(X)$ is the join $\{n \text{ points}\} * X$.
So, for example, $\Sigma_0(X) = X$, $\Sigma_1(X) = \text{Cone}(X)$ and $\Sigma_2(X)$ is the usual suspension of $X$. 
If $\alpha = (n_1, n_2, \ldots,n_k) $ is a sequence of non-negative integers, we denote by $\Sigma_\alpha (X)$ the $\alpha$-suspension of $X$, \emph{i.e.}~the composition $\Sigma_{n_1}\circ\dots\circ\Sigma_{n_k}$ applied to $X$. We also set $\Sigma(\alpha) \coloneqq \Sigma_\alpha (\emptyset)$.

Note that $$\Sigma(\underbrace{1, \ldots,1}_n) = \Delta^{n-1}$$ and
\begin{equation}\label{eq:zeros}
\Sigma_{(\ldots, n_i, 0, n_{i+2}, \ldots)} (X) = \Sigma_{(\ldots, n_i,  n_{i+2}, \ldots)} (X) \ .
\end{equation}
As a consequence of Equation~\eqref{eq:zeros}, from now on we assume that all strings $\alpha$ do not contain any zero entry.
Since the join of complexes is commutative and associative \cite[Lemma~62.4]{munkres}, we see that $\Sigma_{(a,b)}(X) \simeq \Sigma_{(b,a)}(X)$. 
In particular, if $a_1, \ldots, a_k$ are positive integers, then $\Sigma(1, a_1, ... ,a_k)$ is always a cone.

\begin{lem}\label{lem:joinspheres}
Let $m$ be an integer greater than $1$. The $m$-point suspension $\Sigma_{m}(S)$ of a wedge of spheres~$S = \bS^{n_1} \vee \dots \vee \bS^{n_k}$ is homotopy equivalent to 
\[
\bigvee_{i=1}^{m-1} \left( \bS^{n_1+1} \vee \dots \vee \bS^{n_k+1} \right) \ .
\]
\end{lem}
\begin{proof}
Recall that for a sphere $\bS^n$, the join with $m \ge 1$ points is a wedge of $m-1$ spheres:
\[
\bS^n * \{m\text{ points}\} \simeq \bigvee_{i=1}^{m-1}\bS^{n+1} \ .
\]
Since the join and the wedge commute up to homotopy, the assertion follows.
\end{proof}

We obtain the following observation as a straightforward consequence of Lemma~\ref{lem:joinspheres}:
\begin{rem}\label{rem:easyfromlemma}
Let $\alpha = (n_1, n_2, \ldots,n_k)$ be a sequence of positive integers, and assume $k>0$. Then, if $n_i = 1$ for some $i$, $\Sigma(\alpha)$ is contractible. Otherwise, we have a homotopy equivalence: 
$$\Sigma(\alpha) \simeq \bigvee^{q(\alpha)} \bS^{\vert \alpha \vert -1} \ , $$ where $q(\alpha) = \prod_{i} (n_i -1)$, and $|\alpha| = \sum_{i=1}^k n_i$.
As a consequence, $\Sigma(\alpha)$ is simply connected if and only if $k \ge 3$. 
\end{rem}

 This next result implies that the whole homotopy type of the complexes $\mat(\tG)$ is completely determined by the collection of indegrees of the vertices in $\tG$.

\begin{prop}\label{prop:homotopy of or mat}
Let $\tG$ be an oriented graph on $n$ vertices. Then, the oriented matching complex of $\tG$ is either contractible or a wedge of spheres. More precisely,  if there is $v$ such that $\mathrm{indeg}(v) =1$, then $\mat(\tG)$ is contractible. Otherwise, 
$$\mat(\tG) \simeq \bigvee^{q} \bS^{N-1} $$  where $$ q \coloneqq \prod_{\substack{ v \in V(\tG) \\ \mathrm{indeg}(v)>1}} (\mathrm{indeg}(v) -1),$$
and $N$ is the number of vertices whose indegree is positive.
\end{prop}

\begin{proof}
By Proposition~\ref{prop:or_mat=multi_sr} we have the following isomorphism of simplicial complexes 
\[\mat(\tG) \cong X(\tG_{sr}) \ .\]
Observe that $\tG_{sr}$ is the disjoint union of graphs $\tH_v$,  for $v\in V(\tG)$ with $\mathrm{indeg}(v)\geq 1$, each one isomorphic to a sink with $\mathrm{indeg}(v) + 1$ vertices. 
The multipath complex of a disjoint union is the join of the corresponding multipath complexes, and the multipath complex of a sink  with $m$ vertices is the disjoint union of $m$ points -- see \cite[Subsection~6.4]{secondo}. The statement now follows from Remark~\ref{rem:easyfromlemma}.
\end{proof}

A simplicial complex $X$ is \emph{shellable} if there exists an ordering $\Delta_1$, ..., $\Delta_k$ of the maximal simplices in $X$ such that the complex
\[ Y_h = \Delta_h \cap  \bigcup_{i=1}^{h-1} \Delta_i\]
is pure (\emph{i.e.}~all its maximal simplices have the same dimension) of dimension $\mathrm{dim}(\Delta_h) - 1$, for each~$h$.
There is also a related notion of \emph{strong shellability};  we will not recall it here and refer the reader to \cite[Definition~2.2]{MR4015988}. For a simplicial complex, strong shellability implies shellability. 
%Being shellable depends on the simplicial structure, and not only on the homomorphism class of the realisation -- \emph{cf.}~\cite{VINCE198591}.

It is well-known that  shellable (and strongly shellable) complexes are homotopy equivalent to wedges of spheres, one for each maximal simplex (see, for instance, \cite[Theorem~12.3]{Kozlov}).
Therefore, it is natural to ask whether non-contractible oriented matching complexes are shellable. We affermatively answer this question;

\begin{cor}
Oriented matching complexes are either strongly shellable or contractible. \\ In particular, all non-contractible matching complexes are Cohen-Macaulay.
\end{cor}
\begin{proof}
The join of strongly shellable complexes is strongly shellable by \cite[Proposition~2.16]{MR4015988} (see also \cite[Remark~10.22]{MR1401765} for ``non-strong'' shellability). Wedges of $0$-dimensional spheres, that is finite unions of more than one points, are strongly shellable.
The oriented matching complex can be obtained as the iterated join of wedges of $0$-dimensional spheres if and only if there are no vertices with indegree $1$  (\emph{cf.}~the proof of Proposition~\ref{prop:homotopy of or mat}). Since, by Proposition~\ref{prop:homotopy of or mat}, the presence of vertices of indegree $1$ is equivalent to the oriented matching complex being contractible, we obtain the first part of the statement.

Oriented matching complexes are pure; the dimension of each maximal simplex is the number of vertices with positive indegree. It was observed in \cite{MR4015988} that pure strongly shellable complexes are Cohen-Macaulay, giving the second part of the statement.
\end{proof}

Following this last proposition, it is easy to give a complete characterisation of the graphs whose oriented matching complex is homotopically equivalent to a sphere.
\begin{cor}
The oriented matching complex $\mat(\tG)$ of an oriented graph $\tG$ is homotopically equivalent to a sphere $\bS^n$ if and only if all vertices have indegree $0$ or $2$.
\end{cor}
\begin{proof}
Apply Proposition~\ref{prop:homotopy of or mat}. The number of vertices of indegree $2$ gives the dimension of the sphere.
\end{proof}

\subsection{Oriented matching homology with non-constant coefficients}

We proved in Proposition~\ref{prop:or_mat=multi_sr} that the isomorphism between the face poset of $\MM^o(\tG)$ and the path poset of $\tG_{sr}$ induces an isomorphism of the associated monotone cohomologies with constant coefficients. It is however possible to show that this isomorphism cannot be straightforwardly extended to an isomorphism between the respective monotone cohomologies groups with arbitrary functor-valued coefficients.

One easy example is given by the graph $\tG$ shown in Figure~\ref{fig:sourceres}(A); let $A$ be a finitely generated $\bF$-algebra of dimension~$\alpha$. Consider $C_{\mu}^*(\tG;A)$, the multipath cochain complex of  $\tG$ with coefficients in the functor $\cF_A\colon\mathbf{P}(\tG)\to \mathbf{Vect}_{\bF}$  and denote by $C_o^*(\tG;A)$ the monotone cochain complex associated to  oriented matchings (\emph{cf.}~Definition~\ref{def:monhom}). 
The chain complex $C_o^*(\tG;A)$ admits the  following description:
\[
\begin{tikzcd}
\cdots \arrow{r} & 0 \arrow{r} &A^{\otimes 6} \arrow{r} & A^{\otimes 5}\oplus A^{\otimes 5} \oplus A^{\otimes 5}\arrow{r}& A^{\otimes 4}\oplus A^{\otimes 4}  \arrow{r} & 0 \arrow{r} & \cdots
\end{tikzcd}
\]
where $A^{\otimes 6} $ sits in  degree $0$. The Euler characteristic of $C_o^*(\tG;A)$ is easily computed to be $\alpha^4 ( \alpha  -2)(\alpha - 1 ) $. 
On the other hand, the cochain complex associated to~$\tG_{sr}$, using the poset in Figure~\ref{fig:sourceres}(C) is
\[
\begin{tikzcd}
\cdots \arrow{r} & 0 \arrow{r} &A^{\otimes 5} \arrow{r} & A^{\otimes 4}\oplus A^{\otimes 4} \oplus A^{\otimes 4}\arrow{r}& A^{\otimes 3}\oplus A^{\otimes 3}  \arrow{r} & 0 \arrow{r} & \cdots
\end{tikzcd}
\]
where $A^{\otimes 5} $ sits in cohomological degree $0$. This time, the Euler characteristic of $C_o^*(\tG;A)$ is given by $\alpha^3 ( \alpha  -2)(\alpha - 1 ) $. Therefore, these complexes must have distinct homologies for any choice of $A$ of dimension $\alpha\geq 3$.
%\end{example}

Despite being distinct, the above example suggests 
that the multipath cohomology groups of~$\tG_{sr}$ and the oriented matching cohomology groups of~$\tG$ might still be closely related. We will make use of the following general observation.

\begin{rem}\label{rem:samemonotneposets}
Let $\tG$ and $\tG'$ two graphs, and let $P_\mathscr{M}\subseteq SS(\tG)$ and $P_{\mathscr{M}'}\subseteq SS(\tG')$ be posets associated to the monotone properties $\mathscr{M}$ and $\mathscr{M}'$, respectively.
If there exists an isomorphism of posets
\[ \varphi: P_\mathscr{M} \longrightarrow P_{\mathscr{M}'} \]
such that  there is a graph isomorphism $\varphi(\tH) \cong \tH$ for each $\tH \in P_{\mathscr{M}}$, then we have an isomorphism
\[ (C^*_{\mathcal{F}_{A}}(P_{\mathscr{M}}(\tG)), d^*) \cong (C^*_{\mathcal{F}_{A}}(P_{\mathscr{M}'}(\tG')), d^*) \]
of cochain complexes.
\end{rem}

\begin{thm}\label{thm:iso}
Let $\tG$ be an oriented graph. Then, there is an isomorphism of cochain complexes
\[
C_{o}^*(\tG;A)\cong C_{\mu}^*(\tG_{sr};A)\otimes A^{\otimes s}
\]
where $s$ is the number of  vertices of indegree $0$ in $\tG$.
\end{thm}

\begin{proof}
Consider the graph $\tG'_{sr}$ obtained from $\hasse(\tG)$ by removing all ``non-coherent'' edges, \emph{i.e.}~the edges of the form $((v_i,v_j), v_i) \in E(\hasse(\tG))$ -- \emph{cf.}~Definition~\ref{def:omatch}.
Then, $\tG_{sr}'$ is isomorphic to $\tG_{sr}$ together with the disjoint union of $s$ isolated vertices, $s$ being is the number of vertices of indegree $0$ in $\tG$. 

The graph $\tG'_{sr}$ is bipartite, and its orientation is alternating. 
This yields a bijection $\Phi$ between the set of multipaths in $\tG'_{sr}$ and the set of oriented matching in $\tG$. The bijection is such that a multipath is isomorphic, as a graph, to the corresponding oriented matching on $\tG$, thus providing an identification of the two. Moreover, inclusions of multipaths correspond to inclusions of matchings. Therefore, as an application of Remark~\ref{rem:samemonotneposets}, the associated cochain complexes are isomorphic. We conclude by noting that the multipath cochain complex of $\tG'_{sr}$ is isomorphic to the tensor product $ C_{\mu}^*(\tG_{sr};A)\otimes A^{\otimes s}$, since the multipath cochain complex of a disjoint union is the tensor product of the multipath cochain complexes (\emph{cf.}~\cite[Remark~3.2]{secondo}).
\end{proof}

Theorem~\ref{thm:iso} shows how to carry out computations with general functor coefficients; for sake of completeness, we provide an example:

\begin{example}
Consider the $3$-clique graph in Figure~\ref{fig:sourceres}(A), and set $A = \bF[X]/(X^2)$, for some field~$\bF$. By Theorem~\ref{thm:iso}, we know that $C_{o}^*(\tG;A)\cong C_{\mu}^*(\tG_{sr};A)\otimes A^{\otimes s}$, where $s=1$ and $\tG_{sr} = \tA_2 \sqcup \tL_1$ is the graph in  in Figure~\ref{fig:sourceres}(B).
As mentioned in the proof above, the multipath cochain complex of a disjoint union is the tensor product of the multipath cochain complexes. It follows that
\[C_{\mu}^*(\tG_{sr};A) \cong C_{\mu}^*(\tA_{2};A) \otimes C_{\mu}^*(\tL_{1};A)\ ,\]
where the tensor product is a graded tensor product.
Using \cite[Example~4.16]{primo} we can deduce that the multipath cohomology $\mathrm{H}_{\mu}^*(\tA_{2}; A) $ is isomorphic to $\bF^2$ in degree $0$, and $\bF^2$ in degree $1$. It is then easy to see that $\mathrm{H}_{\mu}^*(\tL_{1};A) $ is isomorphic to $\bF^2$ in degree $0$. 
Putting everything together, we see that the cohomology of the cochain complex $C_o^*(\tG;A)$ is
\[{\rm H}_o^i(\tG;A) \cong \begin{cases} \bF^{8} &  i = 0,1,\\ (0) &  \text{otherwise},\end{cases}\]
concluding the computation.
\end{example}

\section{Oriented homology and free-flow orientations}

In this final section, we introduce a special instance of the general constructions outlined in Section~\ref{sec:cohomologies}. This is a novel homology theory for oriented graphs; its construction is formally similar to the homology theories introduced in a different context in \cite{caputi2022categorifying}. 

Let $\tG$ be an unoriented graph. Recall that $O(\tG)$ denotes the set of all possible orientations on~$\tG$. For a fixed orientation  $\oo$, we can define a poset
$\mathcal{O}(\tG,\oo)\coloneqq (O(\tG), <_{\mathfrak{o}})$
where $<_{\oo}$ is the order relation associated to the covering relation $\prec_{\oo}$ defined as follows: given $\oo_1,\oo_2\in O(\tG)$, we say
$ \oo_1 \prec_{\oo} \oo_2$ if and only if the oriented graphs $\tG_{\oo_1}$ and $\tG_{\oo_2}$  differ only on a single edge $e$, and the orientation of~$e$ in $\tG_{\oo_1}$ is the same as in $\tG_{\oo}$. Clearly, we have a poset isomorphism between  $\mathcal{O}(\tG,\oo)$ and  the Boolean poset~$(\wp(|E(\tG)|),\subseteq)$.

Consider the functor $\cF\colon \mathcal{O}(\tG,\oo)\to \mathbf{grVect}_{\bF}$ to graded vector spaces over the field $\bF$, defined on each object $\oo'$ of (the category associated to) $ \mathcal{O}(\tG,\oo)$ as 
\begin{equation}
    \cF(\oo')\coloneqq \bF\langle \mat(\tG_{\oo'}) \rangle 
\end{equation}
where $\bF\langle \mat(\tG_{\oo'})\rangle$ is the $\bF$-vector space spanned by the simplices in the oriented matching complex $\mat(\tG_{\oo'})$, graded by the simplices' dimension.

For each covering relation $ \oo_1 \prec_{\oo} \oo_2$, we set
\begin{equation}
\label{eq: inclusions of oo}
    \cF(\oo_1 \prec_{\oo} \oo_2)\colon \cF(\oo_1)\to \cF(\oo_2)
\end{equation}
to be the map induced by the inclusion $\mat(\tG_{\oo_1}) \cap \mat(\tG_{\oo_2}) \hookrightarrow \mat(\tG_{\oo_2})$. Arguing as in Proposition~\ref{prop:functpreservessq}, this data describes a functor on the category associated to $ \mathcal{O}(\tG,\oo)$.

Now, for any sign assignment $\e$ on $\mathcal{O}(\tG,\oo)$, we consider the associated poset homology as the homology of the associated cochain complex $\mathrm{OC}(\tG,\oo)\coloneqq (C_\cF^*(\mathcal{O}(\tG,\oo)),d^*)$ -- \emph{cf.}~Theorem~\ref{teo: general cohom}. Note that this homology is bigraded; the first degree increases by one under the action of the differential. The second degree, called \emph{simplicial degree}, is just given by the simplices' dimensions. This latter degree is preserved under the differential, making the homology bigraded.

\begin{defn}
We define the oriented homology of $\tG$ with respect to $\oo$, $\mathrm{OH}(\tG_\oo)$, as the homology of the bigraded cochain complex $(\mathrm{OC}(\tG,\oo),d^*)$.
\end{defn}

As an example, consider the linear graph $\tL_2$ and the alternating graph $\tA_2$ -- \emph{cf.}~Figure~\ref{fig:nstep}. First, note that the oriented matching complex of $\tL_2$ is an interval, whereas the oriented matching complex of $\tA_2$ is given by the disjoint union of two points. The posets $\mathcal{O}(\tL_2,\oo)$ and $\mathcal{O}(\tA_2,\oo)$ are two Boolean posets of the same dimension.
The functor $\cF$ defined above, applied to these posets, yields the graded vector spaces shown in Figure~\ref{fig:two orient}, where the homological degrees go from $0$ to $2$ and the simplicial degrees are supported in degree $0$ and $1$. 

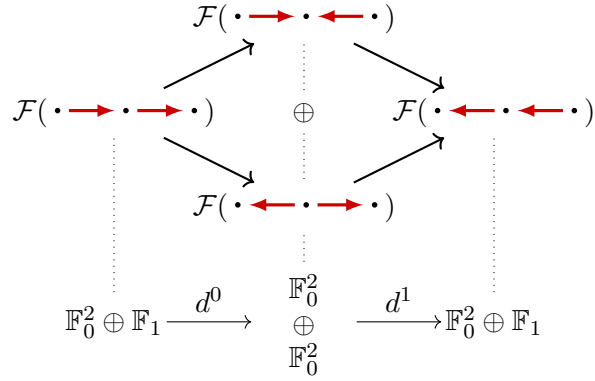
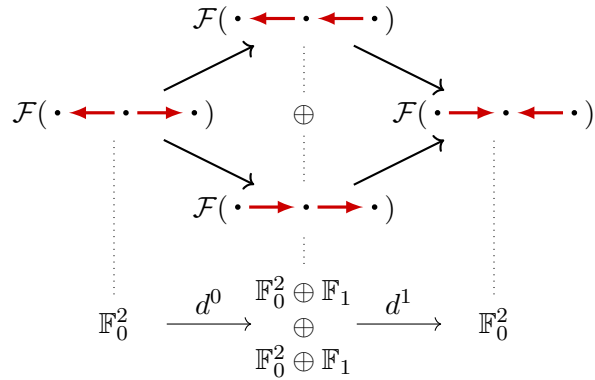
\begin{figure}
\centering
\begin{subfigure}{\textwidth}
\centering   
\begin{tikzpicture}[scale = .5]
\draw[dotted] (0,0) -- (0,-4.85) node[below] {$\bF_{0}^{2}\oplus \bF_1$};

\draw[dotted] (5,3) -- (5,-4) node[below] {$\begin{array}{c}\bF_{0}^{2} \\ \oplus\\ \bF_0^2\end{array}$};

\draw[dotted] (10,0) -- (10,-4.85) node[below] {$\bF_{0}^{2}\oplus \bF_1$};

\node[below] (uc0) at (0,-4.85) {\phantom{$C^0(X)$}};
\node[below] (uc1) at (5,-4.85) {\phantom{$C^0(X)$}};
\node[below] (uc2) at (10,-4.85) {\phantom{$C^0(X)$}};

\draw[->] (uc0) -- (uc1) node[midway, above] {$d^0$};
\draw[->] (uc1) -- (uc2) node[midway, above] {$d^1$};

\node[fill, white] at (5,0){$\oplus$};
\node  at (5,0){$\oplus$};

\node[fill, white] at (0,0) {\phantom{$\cF( \text{\raisebox{0em}{\begin{tikzpicture}[scale =.3, very thick]
    \foreach \a in {1,...,3}
         {
            \node[] (u\a) at ({\a*3},0) {};
            \draw[fill] (u\a) circle (.05);
         }
         \draw[bunired, latex-] (u1) -- (u2);
         \draw[bunired, latex-] (u2) -- (u3);
    \end{tikzpicture}}})$}};

\node[fill, white] at (5,2.5){\phantom{$\cF( \text{\raisebox{0em}{\begin{tikzpicture}[scale =.3, very thick]
    \foreach \a in {1,...,3}
         {
            \node[] (u\a) at ({\a*3},0) {};
            \draw[fill] (u\a) circle (.05);
         }
         \draw[bunired, latex-] (u1) -- (u2);
         \draw[bunired, latex-] (u2) -- (u3);
    \end{tikzpicture}}})$}};
\node[fill, white] at (5,-2.5){\phantom{$\cF( \text{\raisebox{0em}{\begin{tikzpicture}[scale =.3, very thick]
    \foreach \a in {1,...,3}
         {
            \node[] (u\a) at ({\a*3},0) {};
            \draw[fill] (u\a) circle (.05);
         }
         \draw[bunired, latex-] (u1) -- (u2);
         \draw[bunired, latex-] (u2) -- (u3);
    \end{tikzpicture}}})$}};
    
\node[fill, white] at (10,0) {\phantom{$\cF( \text{\raisebox{0em}{\begin{tikzpicture}[scale =.3, very thick]
    \foreach \a in {1,...,3}
         {
            \node[] (u\a) at ({\a*3},0) {};
            \draw[fill] (u\a) circle (.05);
         }
         \draw[bunired, latex-] (u1) -- (u2);
         \draw[bunired, latex-] (u2) -- (u3);
    \end{tikzpicture}}})$}};

\node (a) at (0,0) {$\cF(\text{\raisebox{0em}{\begin{tikzpicture}[scale =.3, very thick]
    \foreach \a in {1,...,3}
         {
            \node[] (u\a) at ({\a*3},3) {};
            \draw[fill] (u\a) circle (.05);
         }
         \draw[bunired, -latex] (u1) -- (u2);
         \draw[bunired, -latex] (u2) -- (u3);
    \end{tikzpicture}}})$};

\node (b1) at (5,2.5) { $\cF(\text{\raisebox{0em}{\begin{tikzpicture}[scale =.3, very thick]
    \foreach \a in {1,...,3}
         {
            \node[] (u\a) at ({\a*3},0) {};
            \draw[fill] (u\a) circle (.05);
         }
         \draw[bunired, -latex] (u1) -- (u2);
         \draw[bunired, latex-] (u2) -- (u3);
    \end{tikzpicture}}})\ \ $};
\node (b2) at (5,-2.5) { $\cF(\text{\raisebox{0em}{\begin{tikzpicture}[scale =.3, very thick]
    \foreach \a in {1,...,3}
         {
            \node[] (u\a) at ({\a*3},0) {};
            \draw[fill] (u\a) circle (.05);
         }
         \draw[bunired, latex-] (u1) -- (u2);
         \draw[bunired, -latex] (u2) -- (u3);
    \end{tikzpicture}}})\ \ $};

\node (c) at (10,0) { $\cF( \text{\raisebox{0em}{\begin{tikzpicture}[scale =.3, very thick]
    \foreach \a in {1,...,3}
         {
            \node[] (u\a) at ({\a*3},0) {};
            \draw[fill] (u\a) circle (.05);
         }
         \draw[bunired, latex-] (u1) -- (u2);
         \draw[bunired, latex-] (u2) -- (u3);
    \end{tikzpicture}}})$};

\draw[thick, ->] (a) -- (b1) node[midway,above,rotate =31] {}; 
\draw[thick, ->] (a) -- (b2) node[midway,above] {}; 

\draw[thick, ->] (b1) -- (c) node[midway,above] {}; 
\draw[thick, ->] (b2) -- (c) node[midway,above left,rotate =-29] {}; 
\end{tikzpicture}
\caption{The Boolean poset $\mathcal{O}(\tL_2)$ and ${\rm OC}(\tL_2)$.}
\end{subfigure}
~\\~\\
\begin{subfigure}{\textwidth}
\centering   
\begin{tikzpicture}[scale = .5]
\draw[dotted] (0,0) -- (0,-4.85) node[below] {$\bF_{0}^{2}$};

\draw[dotted] (5,3) -- (5,-4) node[below] {$\begin{array}{c}\bF_{0}^{2} \oplus \bF_1\\ \oplus\\ \bF_0^2\oplus \bF_1\end{array}$};

\draw[dotted] (10,0) -- (10,-4.85) node[below] {$\bF_{0}^{2}$};

\node[below] (uc0) at (0,-4.85) {\phantom{$C^0(X)$}};
\node[below] (uc1) at (5,-4.85) {\phantom{$C^0(X)$}};
\node[below] (uc2) at (10,-4.85) {\phantom{$C^0(X)$}};

\draw[->] (uc0) -- (uc1) node[midway, above] {$d^0$};
\draw[->] (uc1) -- (uc2) node[midway, above] {$d^1$};

\node[fill, white] at (5,0){$\oplus$};
\node  at (5,0){$\oplus$};

\node[fill, white] at (0,0) {\phantom{$\cF( \text{\raisebox{0em}{\begin{tikzpicture}[scale =.3, very thick]
    \foreach \a in {1,...,3}
         {
            \node[] (u\a) at ({\a*3},0) {};
            \draw[fill] (u\a) circle (.05);
         }
         \draw[bunired, latex-] (u1) -- (u2);
         \draw[bunired, latex-] (u2) -- (u3);
    \end{tikzpicture}}})$}};

\node[fill, white] at (5,2.5){\phantom{$\cF( \text{\raisebox{0em}{\begin{tikzpicture}[scale =.3, very thick]
    \foreach \a in {1,...,3}
         {
            \node[] (u\a) at ({\a*3},0) {};
            \draw[fill] (u\a) circle (.05);
         }
         \draw[bunired, latex-] (u1) -- (u2);
         \draw[bunired, latex-] (u2) -- (u3);
    \end{tikzpicture}}})$}};
\node[fill, white] at (5,-2.5){\phantom{$\cF( \text{\raisebox{0em}{\begin{tikzpicture}[scale =.3, very thick]
    \foreach \a in {1,...,3}
         {
            \node[] (u\a) at ({\a*3},0) {};
            \draw[fill] (u\a) circle (.05);
         }
         \draw[bunired, latex-] (u1) -- (u2);
         \draw[bunired, latex-] (u2) -- (u3);
    \end{tikzpicture}}})$}};
    
\node[fill, white] at (10,0) {\phantom{$\cF( \text{\raisebox{0em}{\begin{tikzpicture}[scale =.3, very thick]
    \foreach \a in {1,...,3}
         {
            \node[] (u\a) at ({\a*3},0) {};
            \draw[fill] (u\a) circle (.05);
         }
         \draw[bunired, latex-] (u1) -- (u2);
         \draw[bunired, latex-] (u2) -- (u3);
    \end{tikzpicture}}})$}};

\node (a) at (0,0) {$\cF(\text{\raisebox{0em}{\begin{tikzpicture}[scale =.3, very thick]
    \foreach \a in {1,...,3}
         {
            \node[] (u\a) at ({\a*3},3) {};
            \draw[fill] (u\a) circle (.05);
         }
         \draw[bunired, latex-] (u1) -- (u2);
         \draw[bunired, -latex] (u2) -- (u3);
    \end{tikzpicture}}})$};

\node (b1) at (5,2.5) { $\cF(\text{\raisebox{0em}{\begin{tikzpicture}[scale =.3, very thick]
    \foreach \a in {1,...,3}
         {
            \node[] (u\a) at ({\a*3},0) {};
            \draw[fill] (u\a) circle (.05);
         }
         \draw[bunired, latex-] (u1) -- (u2);
         \draw[bunired, latex-] (u2) -- (u3);
    \end{tikzpicture}}})\ \ $};
\node (b2) at (5,-2.5) { $\cF(\text{\raisebox{0em}{\begin{tikzpicture}[scale =.3, very thick]
    \foreach \a in {1,...,3}
         {
            \node[] (u\a) at ({\a*3},0) {};
            \draw[fill] (u\a) circle (.05);
         }
         \draw[bunired, -latex] (u1) -- (u2);
         \draw[bunired, -latex] (u2) -- (u3);
    \end{tikzpicture}}})\ \ $};

\node (c) at (10,0) { $\cF( \text{\raisebox{0em}{\begin{tikzpicture}[scale =.3, very thick]
    \foreach \a in {1,...,3}
         {
            \node[] (u\a) at ({\a*3},0) {};
            \draw[fill] (u\a) circle (.05);
         }
         \draw[bunired, -latex] (u1) -- (u2);
         \draw[bunired, latex-] (u2) -- (u3);
    \end{tikzpicture}}})$};

\draw[thick, ->] (a) -- (b1) node[midway,above,rotate =31] {}; 
\draw[thick, ->] (a) -- (b2) node[midway,above] {}; 

\draw[thick, ->] (b1) -- (c) node[midway,above] {}; 
\draw[thick, ->] (b2) -- (c) node[midway,above left,rotate =-29] {}; 
\end{tikzpicture}
\caption{The Boolean poset $\mathcal{O}(\tA_2)$ and ${\rm OC}(\tA_2)$.}
\end{subfigure}
\caption{The complexes ${\rm OC}(\tL_2)$ (above) and ${\rm OC}(\tA_2)$ (below) associated to the linear graph $\tL_2$ and the alternating graph $\tA_2$, respectively. The symbol $\bF^n_b$ denotes a copy of $\bF^n$ in simplicial degree $b$.}
\label{fig:two orient}
\end{figure}
Crucially, the cohomology $\mathrm{OH}$ does depend on the fixed orientation on $\tG$. Indeed, the two graphs just considered correspond to two different orientations on the linear unoriented graph with two edges. A simple computation based on the graded vector spaces in Figure~\ref{fig:two orient} shows that the Euler characteristic of $\mathrm{OH}(\tL_2)$ and $\mathrm{OH}(\tA_2)$ are equal in absolute value, but have different sign; hence, their poset cohomologies are not isomorphic.
Furthermore, if we denote by $\bF^n_{a,b}$ a copy of $\bF^n$ in homological degree $a$ and simplicial degree $b$, explicitly computing the homology for both graphs we obtain that~${\rm OH}^{*,*}(\tL_2) \cong \bF_{0,1} \oplus \bF_{2,1}$, and ${\rm OH}^{*,*}(\tL_2) \cong \bF^2_{1,1}$.\\

We conclude by proving the following unexpected result; namely, the oriented homology $\mathrm{OH}(\tG, \oo)$ admits an extremely simple and explicit description, regardless of the initial orientation~$\oo$ chosen. In particular, we show that $\mathrm{OH}(\tG,\oo)$ is simply a count of the free-flow orientations on~$\tG$.

\begin{thm}\label{thm:freeflows and oriented}
Let $\tG$ be a connected unoriented graph, and let $\oo$ be an orientation on $\tG$. Then, the generators of $\mathrm{OH}^*(\orient)$ are in bijection with  free-flow orientations on $\tG$.
More precisely, $\dim ({\rm OH}^i(\orient))$ is the number of free-flow orientations on $\tG$ obtained from $\oo$ by changing the orientation of exactly $i$ edges.
In particular, ${\rm OH}^*(\orient)$ is non-trivial if and only if $\tG$ is a pseudotree. 
\end{thm}
\begin{proof}
By definition, ${\rm OC}^i(\tG, \oo)$ is the $\bF$-vector space spanned by oriented matchings on $\tG_{\oo'}$, for all orientations $\oo'$ differing from $\oo$ on precisely $i$ edges. 

We call a chain complex \emph{Boolean of dimension $n+1$} if it is isomorphic to the reduced simplicial chain complex of the standard $n$-simplex (the $-1$-simplex being the empty set).
We claim that for each matching $m$ in $\MM(\tG) = \bigcup_{\oo} \mat(\tG_{\oo})$ (\emph{cf.}~Equation~\eqref{eq:matching as union}), we have a Boolean subcomplex $C(m)$ of ${\rm OC}(\tG,\oo)$, and that moreover ${\rm OC}(\tG,\oo)$ splits as the direct sum of these complexes.

Fix $m \in \MM(\tG)$, and define 
\[ C(m) = \bF\langle m' \in \mat(\tG, \oo')\mid \oo'\in \mathcal{O}(\tG),\ m' = m \rangle \subseteq {\rm OC}(\tG,\oo) \ .\]
Since the differential $d$ of ${\rm OC}^*(\tG, \oo)$ is induced by the inclusions $\mat(\tG_{\oo_1}) \cap \mat(\tG_{\oo_2}) \hookrightarrow \mat(\tG_{\oo_2})$ (\emph{cf.}~Equations~\eqref{eq:diff} and~\eqref{eq: inclusions of oo}), it follows that $C(m)$, endowed with its induced differential, is a subcomplex. Furthermore, if $d(x)\in C(m)$ then $x\in C(m)$. This last fact implies that ${\rm OC}(\tG,\oo)$ splits as the direct sum of these complexes.
It is not hard to see that $C(m)$ is Boolean;  first, consider the set $F$ of edges of $\tG$ identified by the matching $m\subset E( \hasse (\tG))$. For an orientation $\oo'$ on $\tG$, $m$ will belong to $\mat (\tG_{\oo'})$ if and only if the target of each edge in $F$ is the same as the corresponding edge in $m$ (\emph{cf.}~Definition~\ref{def:omatch}).
Let $\widehat{\oo}$ be the unique orientation on $\tG$ such that the edges in $F$ are coherent with $m$, and the remaining edges are oriented as in $\oo$. Starting from~$\widehat{\oo}$ and changing the orientations of the edges in $E(\tG)\setminus F$ we get all orientations $\oo'$ such that 
\[ C(m)\cap \bF\langle \mat(\tG_{\oo'}) \rangle  \neq (0) \ . \]
In particular, $C(m)$ is Boolean of dimension $\vert E(\tG)\setminus F\vert$. 

Boolean complexes have trivial homology, unless they are of dimension $0$ -- in which case the homology is free of rank~$1$. Therefore, the generators of ${\rm OH}(\tG_\oo)$ correspond  precisely to the oriented matchings on $\tG$ with the same number of edges as $\tG$. In turn, each of these matchings induce an orientation on $\tG$ where all vertices have indegree $1$ or $0$. Thus, by Lemma~\ref{lem:pseudoindegree} these orientations are exactly the free-flow orientations on $\tG$.
\end{proof}

This last result implies in particular that the rank of the oriented homology is independent of the initial choice of orientation. However, as showed in the previous example and in Figure~\ref{fig:two orient}, the oriented homology of a given graph with respect to two distinct orientations can be different. Namely, the generators might be supported in different degrees.

\bibliographystyle{alpha}

\end{document}